\documentclass[a4paper]{amsart}
\usepackage{amssymb}
\usepackage{amsmath}
\usepackage{booktabs}
\usepackage{enumitem}
\usepackage[T1]{fontenc}
\usepackage{mathtools}
\usepackage{microtype}

\usepackage[all,cmtip]{xy}
\usepackage{tikz}
\usetikzlibrary{matrix,arrows,decorations.pathmorphing}
\tikzset{every picture/.style={line width=0.75pt}}

\usepackage{hyperref}
\usepackage[capitalise]{cleveref}

\hypersetup{
	colorlinks=true,
	linkcolor=blue,
	anchorcolor=blue,
	citecolor=green
}

\theoremstyle{plain}
\newtheorem{thm}{Theorem}[section]

\newtheorem{claim}[thm]{Claim}

\newtheorem{corollary}[thm]{Corollary}
\newtheorem{lemma}[thm]{Lemma}
\newtheorem{proposition}[thm]{Proposition}
\newtheorem{question}[thm]{Question}
\newtheorem{theorem}[thm]{Theorem}

\theoremstyle{definition}
\newtheorem{definition}[thm]{Definition}

\newtheorem{notation}[thm]{Notation}
\newtheorem*{notation*}{Notation and Terminology}
\newtheorem{remark}[thm]{Remark}
\newtheorem{example}[thm]{Example}

\theoremstyle{remark}

\DeclareMathOperator{\Bl}{Bl}
\DeclareMathOperator{\coeff}{coeff}

\DeclareMathOperator{\Exc}{Exc}

\DeclareMathOperator{\length}{length}
\DeclareMathOperator{\N}{N}
\DeclareMathOperator{\NE}{NE}
\DeclareMathOperator{\Nef}{Nef}
\DeclareMathOperator{\NS}{NS}
\DeclareMathOperator{\PE}{PE}
\DeclareMathOperator{\rank}{rank}

\DeclareMathOperator{\Sym}{Sym}
\DeclareMathOperator{\tp}{top}

\makeatletter

\newcommand{\Rmnum}[1]{\expandafter\@slowromancap\romannumeral #1@}
\makeatother

\title[Tangent bundles]
{Smooth projective surfaces with pseudo-effective tangent bundles}

\author{Jia Jia}
\address{Department of Mathematics,
	National University of Singapore,
	Singapore 119076, Republic of Singapore
}
\email{jia{\_}jia@u.nus.edu}

\author{Yongnam Lee}
\address{Center for Complex Geometry,
	Institute for Basic Science (IBS),
	55 Expo-ro, Yuseong-gu, Daejeon, 34126, Republic of Korea, and
	\newline
	\indent Department of Mathematical Sciences, KAIST, 291 Daehak-ro, Yuseong-gu, Deajeon, 34141, Republic of Korea
}
\email{ynlee@kaist.ac.kr, ynlee@ibs.re.kr}

\author{Guolei Zhong}
\address{Center for Complex Geometry,
	Institute for Basic Science (IBS),
	55 Expo-ro, Yuseong-gu, Daejeon, 34126, Republic of Korea
}
\email{zhongguolei@u.nus.edu, guolei@ibs.re.kr}


\begin{document}

\begin{abstract}
	Let \(S\) be a non-uniruled (i.e., non-birationally ruled) smooth projective surface.
	We show that the tangent bundle \(T_S\) is pseudo-effective if and only if
	the canonical divisor \(K_S\) is nef and the second Chern class vanishes,
	i.e., \(c_2(S)=0\).
	Moreover, we study the blow-up of a non-rational ruled surface with pseudo-effective tangent bundle.
\end{abstract}

\maketitle

\section{Introduction}

We work over the field \(\mathbb{C}\) of complex numbers.
A smooth projective variety comes naturally equipped with the tangent bundle,
which is the dual of its sheaf of K\"ahler differentials,
and the related properties of such objects can be applied
to classify algebraic varieties.
A well-known theorem of Mori \cite{Mor79} asserts that
if the tangent bundle of a smooth projective variety is ample,
then it is a projective space,
which gives a solution to Hartshorne's conjecture.
Since then, the study of smooth projective varieties
whose tangent bundles admit some positivity properties
has attracted a lot of attention
and such properties are usually expected to
impose strong restrictions on the geometry of the underlying varieties.

\begin{definition}
	Let \(X\) be a smooth projective variety.
	Given a vector bundle \(E\) on  \(X\),
	we denote by \(\mathbb{P}(E)\) the Grothendieck projectivisation of \(E\)
	with \(\mathcal{O}_{\mathbb{P}(E)}(1)\) denoting the relative hyperplane section bundle.
	Recall that \(E\) is ample (resp.~nef, big, pseudo-effective) if
	\(\mathcal{O}_{\mathbb{P}(E)}(1)\) is ample (resp.~nef, big, pseudo-effective) on \(\mathbb{P}(E)\).
\end{definition}

Following the program of Campana and Peternell, a smooth Fano variety with nef tangent bundle
is conjectured to be a rational homogeneous space,
and this conjecture has been intensively studied (cf.~\cite{CP91,DPS94,MOS15,Kan17}).
Starting from this aspect, it is natural to classify smooth projective varieties
with other positivity properties,
e.g., with big or pseudo-effective tangent bundles.
In the past few years,
there are many beautiful results in this direction, especially
when \(X\) is a Fano manifold, i.e., the anti-canonical divisor \(-K_X\) is ample.
For example, H\"oring, Liu and Shao showed in \cite[Theorem~1.2]{HLS22} that
the tangent bundle of a smooth del Pezzo surface (i.e., a Fano surface) of degree \(d\coloneqq K_X^2\) is big (resp. pseudo-effective)
if and only if \(d\geq 5\) (resp.~\(d\geq 4\)).
We refer readers to \cite{Hsi15, Sha20, Mal21, FL22, HL21,   HLS22,  KKL22,  Kim22, Liu22} and the references therein
for more information involving projective manifolds with big tangent bundles.

As smooth projective varieties with big tangent bundles are known to be uniruled
(cf.~\cite[Corollary~8.6]{Miy87-b} and \cite[Proposition~7.1]{Mal21}),
one may ask if there exist many non-uniruled projective varieties sitting in the ``boundary'',
i.e., admitting a pseudo-effective but non-big tangent bundle.
Apart from the trivial example of abelian varieties,
a product of an abelian variety and any smooth projective variety
becomes another example coming to our mind
(cf.~\cref{lem:pe_inj}).
To the best knowledge of ourselves, up to a finite \'etale cover,
there seems no more other example which has been explored before.
Therefore, we propose the following question to study
whether such varieties of product type are basically the only possibilities.

\begin{question}\label{conj:main}
	Let \(X\) be a non-uniruled smooth projective variety.
	Are the following assertions equivalent?
	\begin{enumerate}[wide=0pt,leftmargin=*]
		\item The tangent bundle \(T_X\) is pseudo-effective;
		\item The top Chern class \(c_{\tp}(X)\) vanishes, and the augmented irregularity \(q^\circ(X)\) does not vanish.
	\end{enumerate}
\end{question}

Here, the \emph{augmented irregularity} \(q^\circ(X)\) of a smooth projective variety \(X\)
is defined to be the supremum of \(q(X')\coloneqq h^1(X',\mathcal{O}_{X'})\)
where \(X'\to X\) runs over all the finite \'etale covers of \(X\) (cf.~\cite[Definition~4.1]{NZ09};
see \cref{rmk-cy-fail}).
Note that the top Chern class of a product variety \(X=Y\times Z\) satisfies
\(c_{\tp}(X)=c_{\tp}(Y)\times c_{\tp}(Z)\).
Hence, by a theorem of Lieberman \cite{Lie78},
if a non-uniruled smooth projective variety admits a global holomorphic vector field,
then up to a finite \'etale cover,
it will split into a product of an abelian variety
and a projective variety admitting no global holomorphic vector field,
in which case, the implication \((1)\Rightarrow (2)\) in \cref{conj:main} follows.
However, in general, the pseudo-effectiveness does not necessarily imply
the existence of any global sections of the tangent bundle;
indeed, it is even not clear to us about the non-vanishing
\(H^0(X,\Sym^mT_X)\neq 0\) for some \(m\)-th symmetric power.

We also note that, different from the ampleness and the nefness of a line bundle,
there is a lack of numerical characterisations on the bigness or pseudo-effectiveness,
which makes our investigation a bit more difficult (cf.~\cref{rmk-comparison}).

The main result of our paper is to give a positive answer to \cref{conj:main}
in dimension \(2\):

\begin{theorem}\label{mainthm:pe_iitaka_fib}
	Let \(S\) be a non-uniruled (i.e., non-birationally ruled) smooth projective surface.
	Then the following assertions are equivalent.
	\begin{enumerate}[wide=0pt,leftmargin=*]
		\item The tangent bundle \(T_S\) is pseudo-effective;
		\item \(S\) is minimal and the second Chern class vanishes, i.e., \(c_2(S)=0\).
	\end{enumerate}
	Moreover, if one of the above equivalent conditions holds,
	then the Kodaira dimension \(\kappa(\mathbb{P}(T_S),\mathcal{O}(1))=1-\kappa(S)\in\{0,1\}\),
	and there is a finite \'etale cover \(S'\to S\) such that \(S'\) is either an abelian surface
	or a product \(E\times F\)
	where \(E\) is an elliptic curve and \(F\) is a smooth curve of genus \(\geq 2\).
\end{theorem}

We refer readers to \cite[Proposition 1.10 and Theorem 1.12]{HP21} for a more general result on projective manifolds with pseudo-effective tangent bundles by using the foliation theory (cf.~\cref{rmk-comparison}).
However, our result here is obtained by applying more elementary tools like surface fibration theory (cf.~\cite{Ser96}), Marayama's elementary transformation method (cf.~\cite{Mar82}) and the stability of cotangent bundles (cf.~\cite{Bog79} and \cite{Eno87}), which are well-known to algebraic geometers.

From \cref{mainthm:pe_iitaka_fib}, the pseudo-effectiveness of the tangent bundle
forces the surface to be minimal,
i.e., the canonical divisor is nef.
However, this is no longer true in the higher dimensional case (cf.~\cref{exa:fail-minimal}).
Besides, as the second Chern class of a smooth projective surface of general type is always positive
(cf.~\cite[\Rmnum{7}, (2.4)~Proposition]{BHPV04}),
our result excludes the possibility of a general type surface having pseudo-effective tangent bundle;
see \cref{prop:kappa_2_tx} and \cref{rmk:generaltype-high} for the case of higher dimensional varieties.
As a consequence of \cref{mainthm:pe_iitaka_fib}, we obtain the following corollary.

\begin{corollary}\label{cor:pe-qeff}
	Let \(S\) be a non-uniruled smooth projective surface.
	If the tangent bundle \(T_S\) is pseudo-effective,
	then there is some integer \(m\) such that \(H^0(S,\Sym^mT_S)\neq 0\);
	in particular, the tautological line bundle of \(\mathbb{P}(T_S)\) is \(\mathbb{Q}\)-linearly equivalent to an effective divisor.
\end{corollary}

Before moving to the second part of this section, we give one remark as kindly pointed out by H\"oring.
\begin{remark}\label{rmk-cy-fail}
	Different from \cref{mainthm:pe_iitaka_fib} in the surface case,
	the implication \((2)\Rightarrow (1)\) in \cref{conj:main} would have a negative answer in \(\dim X\geq 3\)
	if we drop the assumption on the non-vanishing of the augmented irregularity.
	Indeed, there do exist a few smooth (strict) Calabi-Yau threefolds which have the vanishing top Chern classes (cf.~\cite[Fig~1]{KS00}).
	On the other hand, as proved in \cite[Theorem~1.6]{HP19} (cf.~\cite[Corollary~6.5]{Dru18}),
	the tangent bundle of a (strict) Calabi-Yau manifold is never pseudo-effective.
\end{remark}

Kim proved in \cite{Kim22} that
a projective bundle \(\mathbb{P}_C(\mathcal{E})\) over a smooth curve \(C\) has a big tangent bundle
if and only if either \(C\cong\mathbb{P}^1\) or \(\mathcal{E}\) is not semi-stable.
Indeed, as kindly pointed out by Kim, according to the proof of  \cite{Kim22}, 
any projective bundle (of arbitrary rank) over
a smooth projective curve
always has a pseudo-effective tangent bundle; see \cref{prop:ruled_surfaces}.
We slightly summarise this result in \cref{sec:ruled}.

In terms of the non-rational ruled surface,
by analysing the position of a single blow-up,
we show the following proposition.
In particular, compared with the non-uniruled surfaces,
there does exist a non-relatively-minimal non-rational uniruled surface
with pseudo-effective but non-big tangent bundle.

\begin{proposition}\label{prop:uniruled_blowup}
	Let \(f\colon S=\mathbb{P}_C(\mathcal{E})\to C\) be a \(\mathbb{P}^1\)-bundle
	over a smooth curve \(C\) with the genus \(g(C)\geq 1\).
	Suppose the tangent bundle \(T_S\) is pseudo-effective but not big.
	Then the blow-up of \(S\) along a point \(p\) has a pseudo-effective tangent bundle
	if and only if there exist some positive integer \(m\) and some line bundle
	\(\mathcal{L}\) which is numerically equivalent to the relative tangent bundle \(T_{S/C}\)
	such that \(H^0(S,\mathfrak{m}_p^m\otimes\mathcal{L}^{\otimes m})\neq 0\),
	where \(\mathfrak{m}_p\) is the maximal ideal of the local ring \(\mathcal{O}_{S,p}\).
\end{proposition}

Making \cref{prop:uniruled_blowup} as an initial point,
we would like to give a rather clean description of   non-rational uniruled projective surfaces
admitting pseudo-effective but non-big tangent bundles.

We summarise the organisation of our paper.
In \cref{sec:pre}, we prepare some preliminary results for the convenience of later use.
In \cref{sec:kappa_02},
we prove our  \cref{mainthm:pe_iitaka_fib} for the cases \(\kappa(S)=0\) and \(2\).
In \cref{sec:elliptic_fib}, we study the case \(\kappa(S)=1\) when \(S\) is minimal,
and the minimality of such \(S\) with pseudo-effective tangent bundle will be shown in \cref{sec-pf-main}.
Finally, we
prove \cref{prop:uniruled_blowup} in \cref{sec:ruled}.

Let us end up the introduction with the following remark.
\begin{remark}[Comparison with previous papers]\label{rmk-comparison}
	In \cite[Theorem~1.1]{Mat22} (cf.~\cite{HIM22}), Matsumura nicely establishes the minimal model program of a projective klt variety
	with strongly pseudo-effective (reflexive) tangent sheaf
	and ends up with a quasi-\'etale quotient of an abelian variety.
	Here, a vector bundle \(E\) on a smooth projective variety \(X\)
	is said to be strongly pseudo-effective
	if \(\mathcal{O}_{\mathbb{P}(E)}(1)\) is pseudo-effective
	and the restricted base locus does not dominate \(X\)
	(cf.~\cite[Definition~7.1]{BDPP13}).
	Since our assumption is weaker than the assumption in \cite{Mat22},
	even in the surface case, we could not expect to obtain a similar result like \cite[Theorem~1.1]{Mat22}.
	The main reason is that,
	when running the minimal model program,
	although our (weak) pseudo-effectiveness of the tangent bundle
	descends along any birational contraction or flips (cf.~\cref{lem:pe_inj}),
	it could not be preserved under a Fano contraction any more
	(see \cite{Kim22} and \cref{prop:uniruled_blowup}; cf.~\cite[Proposition~3.2]{Mat22}).

    After the present paper was finished, we are informed by Professor Andreas H\"oring on his previous joint work \cite{HP21} which gives a nice splitting structure of the tangent bundles of non-uniruled projective manifolds which are pseudo-effective by using a deep foliation theory (cf.~\cite[Proposition 1.10]{HP21}).
    As a result, our \cref{conj:main} has a positive answer in lower dimensional case. 
    Compared with \cite{HP21}, the proof of our result here is more elementary and more self-contained to algebraic-geometers.
\end{remark}

\subsubsection*{\textbf{\textup{Acknowledgements}}}
We would like to thank Professor Andreas H\"oring for informing us his previous joint work \cite{HP21} and the valuable comments.
We would also like to thank Doctor Jeong-Seop Kim  for the valuable comments and suggestions.
The first author is supported by a President's Graduate Scholarship from NUS.
The second and the third authors are supported by the Institute for Basic Science (IBS-R032-D1-2022-a00).

\section{Preliminary}\label{sec:pre}

First of all, we fix the following notation throughout the paper.

\begin{notation}
	Let \(X\) be a projective variety.
	\begin{enumerate}[wide=0pt,leftmargin=*]
		\item The symbol \(\sim\) (resp.~\(\sim_{\mathbb Q}\), \(\equiv\)) denotes
		      the \emph{linear equivalence} (resp.~\emph{\(\mathbb{Q}\)-linear equivalence}, \emph{numerical equivalence})
		      on Cartier divisors
		      (resp.~\(\mathbb{Q}\)-Cartier divisors).
		      Let \(f\colon X\to Y\) be a morphism of projective varieties.
		      We denote by \(\sim_f\)  the \emph{relative (or \(f\)-) linear equivalence} of Cartier divisors,
		      i.e., for two Cartier divisors \(D_1\) and \(D_2\) on \(X\),
		      \(D_1\sim_f D_2\) if and only if there is some Cartier divisor \(E\) on \(Y\) such that \(D_1-D_2\sim f^{*}E\).
		\item Denote by \(\NS(X)\)  the \emph{N\'eron-Severi group} of \(X\).
		      Let \(\N^1(X)\coloneqq\NS(X)\otimes_{\mathbb{Z}}\mathbb{R}\) be the space of \(\mathbb{R}\)-Cartier divisors
		      modulo numerical equivalence and \(\rho(X)\coloneqq\dim_{\mathbb{R}}\N^1(X)\)
		      the \emph{Picard number} of \(X\).
		      Let \(\N_1(X)\) be the dual space of \(\N^1(X)\) consisting of 1-cycles.
		      Denote by \(\Nef(X)\) (resp.~\(\PE(X)\)) the cone of \emph{nef divisors} (resp.~\emph{pseudo-effective divisors})
		      in \(\N^1(X)\) and \(\overline{\NE}(X)\) the dual cone consisting of \emph{pseudo-effective 1-cycles} in \(\N_1(X)\).
		      In particular, when \(X\) is a smooth projective surface,
		      we have \(\N_1(X)=\N^1(X)\) and \(\overline{\NE}(X)=\PE(X)\).
		\item For a smooth projective variety \(X\),
		      we denote by \(K_X\) the \emph{canonical divisor} and \(\kappa(X)=\kappa(X,K_X)\)
		      the \emph{Kodaira dimension} of \(X\).
		\item Let \(f\colon X\to C\) be a surjective morphism between normal projective varieties.
		      We say that \(f\) is a \emph{fibration} if \(f_{*}\mathcal{O}_X=\mathcal{O}_C\)
		      or equivalently, the general fibre of \(f\) is connected.
		      A fibration is said to be \emph{isotrivial}
		      if all the smooth fibres are isomorphic to each other;
		      otherwise, we say that it is \emph{non-isotrivial}.
		      We say that \(f\) is \emph{trivial}
		      if there exist another projective variety \(F\) and an isomorphism \(X\cong F\times C\)
		      such that \(f\) is the natural projection.
		      We say that \(f\) is \emph{locally trivial} (or a \emph{fibre bundle})
		      if each point \(c\in C\) is contained in a small open neighbourhood \(U\)
		      having the property that \(f^{-1}(U)\) is trivial over \(U\).
		      See \cref{note:fibration} for a more detailed description on surface fibrations.
		\item Let \(H\) be a nef and big divisor on \(X\).
		      Let \(\mathcal{F}\) be a torsion free coherent sheaf on \(X\).
		      The \emph{slope} of \(\mathcal{F}\) with respect to \(H\) is defined to be the rational number
		      \[
			      \mu_H(\mathcal{F})\coloneqq\frac{c_1(\mathcal{F})\cdot H^{\dim (X)-1}}{\rank(\mathcal{F})}
		      \]
		      where \(c_1\) is the first Chern class.
		      A torsion free coherent sheaf \(\mathcal{E}\) is said to be \emph{\(\mu\)-semi-stable}
		      if for any non-zero subsheaf \(\mathcal{F}\subseteq\mathcal{E}\),
		      the slopes satisfy the inequality \(\mu_H(\mathcal{F})\leq\mu_H(\mathcal{E})\).
	\end{enumerate}
\end{notation}

Next, we recall and develop some basic properties on the fibration of surfaces.
Most of them come from \cite[Section~3]{Ser96} and \cite[Section~5]{HP20}.

\begin{notation}[Surface fibration]\label{note:fibration}
	\leavevmode%
	\begin{enumerate}[wide=0pt,leftmargin=*]
		\item Let \(S\) be a smooth projective surface and \(f\colon S\to C\) a fibration with \(F\)  a general fibre of \(f\).
		\item For a closed point \(c\in C\),
		      we denote by \(S_{c}\) its fibre over \(c\).
		\item Let \(W\coloneqq \mathbb{P}(T_S)\) be the Grothendieck projectivisation of the tangent bundle of \(S\) with \(\pi\colon W\to S\) the natural projection,
		      and let \(\xi\coloneqq \mathcal{O}_{W}(1)\) be the corresponding tautological class.
		\item Let \(\{\nu_{i}E_{i}\}_{i\in I}\) be the set of all components of non-multiple fibres of \(f\),
		      where the \(\nu_{i}\)'s denote their corresponding multiplicities within their fibres.
		      Let \(\{m_{j}F_{j}\}_{j\in J}\) be the set of multiple fibres of \(f\).
		      Clearly, for each \(j\in J\) and \(c\in C\),
		      we have \(m_jF_j\equiv\sum_{i\in I, f(E_i)=c}\nu_iE_i\equiv F\).
		\item Let
		      \[
			      E=E_S \coloneqq \sum_{c\in C}f^{*}c-(f^{*}c)_{\textup{red}}
			      = \sum_{i\in I} (\nu_{i}-1)E_{i} + \sum_{j\in J} (m_{j}-1)F_{j}
			      = E_{0} + \sum_{j\in J} (m_{j}-1)F_{j}
		      \]
		      where
		      \(E_0\coloneqq\sum_{i\in I} (\nu_i-1)E_i\)
		      comes from the non-multiple non-reduced fibres.
		      This decomposition is indeed the Zariski decomposition
		      where \( \sum_{j\in J} (m_{j}-1)F_{j}\) is the nef part and \(E_{0}\) is the fixed part.
		\item Applying \cite[Section~3]{Ser96},
		      we have an exact sequence
		      \[
			      0\longrightarrow T_{S/C}\longrightarrow T_{S}\longrightarrow f^{*}T_{C}\longrightarrow\mathcal{F}\longrightarrow 0,
		      \]
		      where \(T_{S}\), \(T_{C}\) are the tangent bundles of \(S\) and \(C\) respectively, \(\mathcal{F}\) is a torsion sheaf
		      and \(T_{S/C}\coloneqq \Omega_{S/C}^{\vee}\) is the relative tangent sheaf of \(f\), which is locally free.
		      Let \(J_{S/C}\) be the torsion-free image of \(T_{S}\to f^{*}T_{C}\).
		\item By \cite[(3.0.3) and Proposition~3.1]{Ser96},
		      we have
		      \[
			      J_{S/C}=J_{S/C}^{\vee\vee}\otimes\mathcal{I}_\Gamma
			      =K_{S}^{-1}\otimes T_{S/C}^{-1}\otimes\mathcal{I}_{\Gamma}
		      \]
		      where \(\mathcal{I}_\Gamma\) is an ideal sheaf and the support of \(\Gamma\) consists of points \(s\in S\)
		      such that the reduced structure \(f^{-1}(f(s))_{\textup{red}}\) is singular at \(s\).
		      In addition, we have
		      \(T_{S/C}=-K_{S}+f^{*}K_{C}+E\)
		      where \(E\) is defined in (5).
		      In particular, we have the following short exact sequence
		      \begin{equation}\label{eq:blowup_gamma}
			      \tag{\dag}
			      0\longrightarrow T_{S/C} \longrightarrow T_{S}
			      \longrightarrow (-f^{*}K_{C}-E)\otimes\mathcal{I}_\Gamma \longrightarrow 0.
		      \end{equation}
		\item Let
		      \(Y\coloneqq \mathbb{P}_S((-f^{*}K_C-E)\otimes\mathcal{I}_\Gamma)\cong \mathbb{P}_S(\mathcal{I}_{\Gamma})\subseteq W\)
		      which is a prime divisor on \(W\).
		      Note that \(Y\) is isomorphic to the blow-up of \(S\) along the ideal sheaf \(\mathcal{I}_\Gamma\)
		      since \(\Gamma\) is locally generated by a regular sequence
		      (cf.~\cite[Proposition~3.1 (iii)]{Ser96} and \cite[\S~3.10]{HP20}).
		      However, since the \(\length(\mathcal{O}_{\Gamma,s})\) at each point \(s\in \Gamma\) is not necessarily \(1\),
		      such \(Y\) in general is not necessarily smooth.
		\item Denote by \(\Exc(\pi|_Y)\) the exceptional divisor of \(\pi|_Y\).
		      By the short exact sequence \eqref{eq:blowup_gamma} and \cite[Chapter~\Rmnum{2}, Lemma~7.9 and Proposition~7.13]{Har77},
		      we have \(\xi|_{Y}=\mathcal{O}_{W}(1)|_{Y}=\mathcal{O}_{Y}(1)\)
		      where
		      \[\mathcal{O}_Y(1)=\mathcal{O}_Y(-\Exc(\pi|_Y))\otimes (\pi|_Y)^{*}(-f^{*}K_C-E).\]
		      Moreover, from \cite[\Rmnum{4}, (10.5)~Lemma]{BHPV04} and its proof,
		      we know that \(\xi-Y\sim\pi^{*}T_{S/C}\)
		      and \((\pi|_Y)_{*}\mathcal{O}_Y(1)=(-f^{*}K_C-E)\otimes\mathcal{I}_{\Gamma}\).
	\end{enumerate}
\end{notation}

Now, we recall Maruyama's elementary transformation which will be heavily used in our proofs.
\begin{notation}[Elementary transformation for projective bundles over surface blow-ups]\label{note:elementary_transform}
	\leavevmode%
	\begin{enumerate}[wide=0pt,leftmargin=*]
		\item Let \(h\colon S_2\to S_1\) be a blow-up between smooth projective surfaces
		      with the exceptional \((-1)\)-curve \(i\colon D\hookrightarrow S_2\).
		      Let \(\xi_i\) be the tautological divisor of the projective bundle \(\mathbb{P}(T_{S_i})\)
		      and \(\pi_i\colon \mathbb{P}(T_{S_i})\to S_i\)  the natural projection.
		\item There is a natural short exact sequence
		      \[
			      0\longrightarrow T_{S_2}\longrightarrow h^{*}T_{S_1} \longrightarrow i_{*}T_{D}(D)\longrightarrow 0.
		      \]
		      Here, \(T_D(D)\simeq \mathcal{O}_D(1)\) and \(\widetilde{D}\coloneqq\mathbb{P}_D(T_D(D))\)
		      is a projective subbundle of \(D'\coloneqq\widetilde{\pi_1}^{*}D=\mathbb{P}_D(h^{*}T_{S_1})\)
		      defined by the restriction of the following exact sequence
		      \[
			      0\to K=\mathcal{O}_D(-1) \to h^{*}T_{S_1}|_{D}\to \mathcal{O}_D(1)\to 0,
		      \]
		      where the kernel \(K=\det (h^{*}T_{S_1}|_D)\otimes \mathcal{O}_D(1)^{\vee}=\mathcal{O}_D(-1)\) by the projection formula.
		\item Applying Maruyama's elementary transformation (cf.~\cite[Theorem~1.4 and (1.7)]{Mar82}),
		      we have the following commutative diagram,
		      where \(\beta\) is the blow-up of \(\mathbb{P}(h^{*}T_{S_1})\) along \(\widetilde{D}\)
		      and \(\alpha\) is the blow-down of \(\Bl_{\widetilde{D}}(\mathbb{P}(h^{*}T_{S_1}))\)
		      along the \(\beta\)-strict transform of \(D'=\widetilde{\pi_1}^{*}D\).
		      \[
			      \xymatrix{
			      & & \Bl_{\widetilde{D}}(\mathbb{P}(h^{*}T_{S_1}))\ar[dl]_{\beta}\ar[dr]^{\alpha}\\
			      \mathbb{P}(T_{S_1})\ar[d]^{\pi_1}
			      &\mathbb{P}(h^{*}T_{S_1})\ar@{-->}[rr]\ar[d]^{\widetilde{\pi_1}}\ar[l]_{\widetilde{h}}
			      & & \mathbb{P}(T_{S_2})\ar[d]^{\pi_2}\\
			      S_1
			      & S_2\ar@{=}[rr]\ar[l]_{h}
			      & & S_2
			      }
		      \]
		\item By \cite[Theorem~1.4 and (1.7)]{Mar82}, we have
		      \(\beta^{*}\widetilde{h}^{*}\xi_1\sim\alpha^{*}\xi_2+G\)  where \(G\) is the \(\beta\)-exceptional divisor.
		\item The \(\beta\)-blown-up section \(\widetilde{D}\) satisfies \(\widetilde{D}\sim c_0+e\)
		      in \(D'\)
		      where \(c_0=\widetilde{h}^{*}\xi_1|_{D'}\) is a tautological divisor of \(D'\to D\)
		      and \(e\) is a fibre (cf.~\cite[Chapter V, Proposition 2.6]{Har77}).
		      In particular, \(\widetilde{D}\not\subseteq\widetilde{h}^{*}Y_1\)
		      for any section \(Y_1\in |\xi_1|\) by the above linear equivalence.
	\end{enumerate}
\end{notation}

\medskip

In what follows, we collect several results to be used in the subsequent sections.

\begin{lemma}[{\cite[Lemma~2.2]{HLS22}; cf.~\cite[Lemma 2.7]{Dru18}}]\label{lem:pe_nonvanishing}
	Let \(X\) be a projective variety,
	\(\mathcal{E}\) a vector bundle on \(X\), and
	\(H\) a big \(\mathbb{Q}\)-Cartier \(\mathbb{Q}\)-divisor on \(X\).
	Then \(\mathcal{E}\) is pseudo-effective if and only if for all \(c>0\)
	there exist sufficiently divisible integers \(i,j\in \mathbb{N}\) such that \(i>cj\) and
	\[
		H^{0}(X,\Sym^i\mathcal{E}\otimes\mathcal{O}_{X}(jH))\neq 0.
	\]
\end{lemma}

\begin{lemma}\label{lem:pe_inj}
	Let \(\mathcal{E}\subseteq \mathcal{F}\) be an injection between two vector bundles over a projective variety \(X\).
	If \(\mathcal{E}\) is pseudo-effective, then so is \(\mathcal{F}\).
\end{lemma}

\begin{proof}
	Fix a big \(\mathbb{Q}\)-Cartier \(\mathbb{Q}\)-divisor \(H\) on \(X\).
	By \cref{lem:pe_nonvanishing}, we only need to show for all \(c>0\),
	there exist sufficiently divisible integers \(i,j\in \mathbb{N}\) such that \(i>cj\) and
	\[
		H^{0}(X,\Sym^i\mathcal{F}\otimes\mathcal{O}_{X}(jH))\neq 0.
	\]
	This follows from the injection  \(\Sym^i\mathcal{E}\subseteq \Sym^i\mathcal{F}\)
	and the pseudo-effectiveness of \(\mathcal{E}\) (cf.~\cref{lem:pe_nonvanishing}).
\end{proof}

As a consequence of the above lemma,
a typical difference between nefness and pseudo-effectiveness is that
the quotient bundle of a pseudo-effective vector bundle is not pseudo-effective any more.
For example, a rank two vector bundle \(\mathcal{E}=\mathcal{O}\oplus\mathcal{O}(-1)\) over a smooth rational curve
is pseudo-effective by \cref{lem:pe_inj} while its quotient \(\mathcal{O}(-1)\) is not (cf.~\cite[Proposition~3.4]{Mat22}).

The following example shows that we could not expect
the variety equipped with pseudo-effective tangent bundle to be minimal in the higher dimensional case.
\begin{example}\label{exa:fail-minimal}
	Let \(X\coloneqq E\times S\) be a product of an elliptic curve \(E\)
	and a non-minimal smooth projective surface \(S\)
	which contains some \((-1)\)-curve.
	Let \(p\colon X\to E\) be the natural projection.
	Applying \cref{lem:pe_inj} and considering the natural injection \(0\to p^{*}\mathcal{O}_E\to T_X\),
	we see that
	\(T_X\) is pseudo-effective.
	However, it is clear that \(K_X\) is not nef.
\end{example}

\begin{lemma}[{cf.~\cite[Corollary 2.4]{HLS22}}]\label{lem:bir_descend}
	Let \(\pi\colon X'\to X\) be a birational morphism between smooth projective varieties.
	If the tangent bundle \(T_{X'}\) is pseudo-effective, then so is \(T_X\).
\end{lemma}

We end up this section with the following lemma,
which will be used to confirm the existence of sections of some symmetric power
\(\textup{Sym}^mT_S\) in \cref{mainthm:pe_iitaka_fib} by taking an \'etale base change.
\begin{lemma}[{cf.~\cite[Theorem~5.13]{Uen75}}]\label{lem:ueno-kappa}
	Let \(f\colon X\to Y\) be a surjective morphism of projective varieties
	and \(D\) a Cartier divisor on \(Y\).
	Then \(\kappa(Y,D)=\kappa(X,f^{*}D)\).
\end{lemma}

\section{The case \texorpdfstring{\(\kappa(S)=0\) or \(2\)}{kappa(S)=0,2}}\label{sec:kappa_02}

In this section, we prove \cref{mainthm:pe_iitaka_fib}
when the Kodaira dimension \(\kappa(S)=0\) or \(2\).
Let us begin with the minimal surface of Kodaira dimension zero.

\begin{lemma}\label{lem:minimal_kappa_0}
	Let \(S\) be a smooth minimal projective surface with \(\kappa(S)=0\).
	Then the tangent bundle \(T_{S}\) is pseudo-effective if and only if
	\(S\) is a \textit{Q}-abelian surface
	(and thus the second Chern class \(c_2(S)=0\)),
	i.e., \(S\) is a finite \'etale quotient of an abelian surface.
	In this case, \(\kappa(\mathbb{P}(T_S),\mathcal{O}(1))=1\).
\end{lemma}

\begin{proof}
	From the abundance, we know that \(K_S\equiv 0\).
	Applying \cite[\Rmnum{6}, Table~10]{BHPV04} to \(S\),
	we see that \(S\) has to be one of the following:
	an Enriques surface, a bi-elliptic surface, a K3 surface or an abelian surface.
	Since the tangent bundles of Enriques surfaces and K3 surfaces are not pseudo-effective
	(cf.~\cite[Chapter~\Rmnum{6}, Theorem~4.15]{Nak04}, or more generally, \cite[Theorem~1.6]{HP19}),
	our \(S\) is covered by an abelian surface.
	Conversely, for a finite \'etale morphism \(\pi\colon A\to S\)
	from an abelian surface \(A\),
	we have \(\pi^{*}T_S=T_A=\mathcal{O}_A^{\oplus 2}\).
	Then \(T_S\) is pseudo-effective (and even nef),
	noting that there is an induced \'etale morphism
	\(\widetilde{\pi}\colon\mathbb{P}(\pi^{*}T_S=T_A)\to \mathbb{P}(T_S)\) such that
	\(\widetilde{\pi}^{*}\mathcal{O}_{\mathbb{P}(T_S)}(1)=\mathcal{O}_{\mathbb{P}(T_A)}(1)\) is effective;
	in particular, \(\mathcal{O}_{\mathbb{P}(T_S)}(1)\) and hence \(T_S\) are pseudo-effective
	(cf.~e.g.~\cite[Chapter~\Rmnum{2}, Lemma~5.6]{Nak04}).
	So the first half our lemma is proved.
	The second half of our lemma follows immediately from \cref{lem:ueno-kappa}
	and the fact that \(h^0(A,\Sym^mT_A)=m+1\) for each positive integer \(m\).
\end{proof}

Now we show the minimality of a smooth projective surface
which is of Kodaira dimension zero
and has pseudo-effective tangent bundle.

\begin{proposition}\label{prop:kappa_0_pe_minimal}
	Let \(S\) be a smooth projective surface of \(\kappa(S)=0\).
	If \(T_S\) is pseudo-effective, then \(K_S\) is nef,
	i.e., \(S\) is minimal;
	in particular, \(S\) is an \'etale quotient of an abelian surface.
\end{proposition}

\begin{proof}
	In the view of \cref{lem:bir_descend},
	we only need to exclude the case when \(S\coloneqq S_2\) is a blow-up
	of a smooth minimal surface \(S_1\) with \(K_{S_1}\equiv 0\).
	Suppose to the contrary that the tangent bundle \(T_{S_2}\) is pseudo-effective.
	By \cref{lem:minimal_kappa_0},
	there is an \'etale cover \(S_1'\to S_1\) from an abelian surface \(S_1'\).
	Let \(S_2'\coloneqq S_1'\times_{S_1}S_2\) be the fibre product,
	which is also a smooth projective surface with \(\kappa(S_2')=0\).
	Note that \(T_{S_2'}\) (as the pullback of \(T_{S_2}\)) is also pseudo-effective.
	Replacing \(S_i\) with \(S_i'\),
	we may assume that \(S_1\) is an abelian surface and thus \(T_{S_1}\simeq \mathcal{O}_{S_1}^{\oplus 2}\).

	Let us consider the elementary transformation in \cref{note:elementary_transform} and use the notation therein.
	Note that the tautological divisors \(\xi_1\) and \(\widetilde{h}^{*}\xi_1\)
	can be taken as irreducible horizontal sections \(Y_1\) and \(\widetilde{Y_1}\),
	and both of them are nef.
	Note also that \(Y_1|_{Y_1}\equiv 0\) and \(\widetilde{Y_1}|_{\widetilde{Y_1}}\equiv 0\),
	since \(Y_1\) (resp.~\(\widetilde{Y_1}\)) is a fibre of
	\(\mathbb{P}(T_{S_1})\cong S_1\times\mathbb{P}^1\to \mathbb{P}^1\)
	(resp.~\(\mathbb{P}(h^{*}T_{S_1})\cong S_2\times\mathbb{P}^1\to \mathbb{P}^1\)).
	Moreover, \(\xi_2\) is pseudo-effective if and only if
	so is \(\alpha^{*}\xi_2=\beta^{*}\widetilde{h}^{*}\xi_1-G\)
	(cf.~e.g.~\cite[Chapter~II, Lemma~5.6]{Nak04}).
	On the other hand,
	since \(\beta^{*}\widetilde{h}^{*}\xi_1-G\) is pseudo-effective and \(\widetilde{h}^{*}\xi_1\) is nef,
	by taking a sufficiently ample divisor \(A\) on \(\mathbb{P}(h^{*}T_{S_1})\) such that
	\(\beta^{*}A-G\) is ample (cf.~\cite[Proposition 1.45]{KM98}), we have
	\begin{align*}
		0 & \leq (\beta^{*}\widetilde{h}^{*}\xi_1-G)\cdot(\beta^{*}A-G)\cdot\beta^{*}\widetilde{h}^{*}\xi_1                                                     \\
		  & = (\beta^{*}\widetilde{h}^{*}\xi_1\cdot\beta^{*}A-\beta^{*}\widetilde{h}^{*}\xi_1\cdot G-G\cdot\beta^{*}A+G^2)\cdot \beta^{*}\widetilde{h}^{*}\xi_1 \\
		  & = G^2\cdot\beta^{*}\widetilde{h}^{*}\xi_1=\beta_{*}(G|_G)\cdot \widetilde{h}^{*}\xi_1.
	\end{align*}
	Recall that \(G\to \widetilde{D}\) is a ruled surface and \(G|_G\sim_\beta\mathcal{O}_G(1)\sim_\beta -C_0\)
	where \(C_0\) is some horizontal section, noting that \(G|_G\) may contain some \(\beta\)-fibre component.
	In particular, \(G^2\cdot \beta^{*}\widetilde{h}^{*}\xi_1=-\widetilde{h}^{*}\xi_1\cdot \widetilde{D}=-1<0\) (cf.~\cref{note:elementary_transform}),
	which gives us the desired contradiction.
	The second half follows from \cref{lem:minimal_kappa_0}.
\end{proof}


In the higher dimensional case, applying the Beauville-Bogomolov decomposition,
we can easily see that for a projective manifold \(X\) with \(K_X\equiv 0\),
the tangent bundle \(T_X\) is pseudo-effective if and only if the augmented irregularity \(q^{\circ}(X)>0\) (cf.~\cref{conj:main}).
However, the minimality is not preserved any longer as we observed in \cref{exa:fail-minimal}.

In the second part of this section, we deal with the case when the variety is of general type.
\begin{proposition}\label{prop:kappa_2_tx}
	Let \(S\) be a smooth projective surface of general type.
	Then \(T_S\) is not pseudo-effective.
\end{proposition}

\begin{proof}
	By \cref{lem:bir_descend},
	we only need to exclude the case when
	\(S\) is minimal (and hence \(K_S\) is nef) and has pseudo-effective tangent bundle.
	By the semi-stability of the cotangent bundle \(\Omega_S\)
	(cf.~\cite[Section~13.1]{Bog79}),
	the tangent bundle \(T_S\) is also semi-stable with respect to \(K_S\) (cf.~\cite[Corollary~1.2]{Eno87}).
	We claim that
	\[
		H^0(S,\Sym^iT_S\otimes\mathcal{O}_S(jK_S))=0,
	\]
	for any \(i>2j\).
	Suppose the claim for the time being.
	Then it follows from \cref{lem:pe_nonvanishing} that \(T_S\) is not pseudo-effective which concludes our proposition.

	To prove the claim, suppose to the contrary that \(H^0(S,\Sym^iT_S\otimes\mathcal{O}_S(jK_S))\neq 0\) for some \(i>2j\).
	Fix a non-zero section \(s\) of \(\Sym^iT_S\otimes\mathcal{O}_S(jK_S)\) which defines an injection \(0\to\mathcal{O}_S\to\Sym^iT_S\otimes\mathcal{O}_S(jK_S)\).
	On the other hand, we have
	\begin{align*}
		c_1(\Sym^iT_S\otimes\mathcal{O}_S(jK_S)) & = c_1(\Sym^iT_S) + \rank(\Sym^iT_S)\cdot c_1(\mathcal{O}_S(jK_S)) \\
		                                         & = \frac{i(i+1)}{2}\cdot c_1(T_S)-(i+1)\cdot j\cdot c_1(T_S)       \\
		                                         & = (i+1)\cdot\left(\frac{i}{2}-j\right)c_1(S)
	\end{align*}
	and hence \(c_1(\Sym^iT_S\otimes\mathcal{O}_S(jK_S))\cdot c_1(K_S)<0\).
	This contradicts the semi-stability of \(T_S\) and our claim is thus proved.
\end{proof}

Indeed, with the same argument as above, we obtain the following result in higher dimensional cases.
\begin{remark}\label{rmk:generaltype-high}
	Let \(X\) be a normal (\(\mathbb{Q}\)-factorial) projective variety which is of general type and has at worst klt singularities.
	We shall show that the reflexive tangent sheaf \(T_X\coloneqq\Omega_X^{\vee}\) is not pseudo-effective
	in the sense of \cite[Definition~3.5]{HP20}.
	In the view of \cite[Corollary~4.3]{HP20} and \cref{lem:pe_nonvanishing},
	after running the special minimal model program with scaling (cf.~\cite{BCHM10}),
	we may assume that \(X\) is minimal,
	i.e., \(K_X\) is nef and big.
	By the semi-stability of the tangent sheaf \(T_X\)
	with respect to \(K_X\) (cf.~\cite[Corollary~1.2]{Eno87}),
	we have
	\[
		H^0(X,\Sym^{[i]}T_X\otimes\mathcal{O}_X(jK_X))=0,
	\]
	for \(i>dj\),
	where \(d\coloneqq \dim (X)\) and \(\Sym^{[i]}T_X\) is the reflexive hull of the symmetric power.
	By \cite[Definition~3.5]{HP20} and a similar calculation as in \cref{prop:kappa_2_tx},  \(T_X\) is not pseudo-effective.
	Consequently, the tangent bundle of any hyperbolic smooth projective variety is never pseudo-effective.
\end{remark}

\section{Relatively minimal elliptic fibration, the case \texorpdfstring{\(\kappa(S)=1\)}{kappa(S)=1}}\label{sec:elliptic_fib}

In this section, we study \cref{mainthm:pe_iitaka_fib} for the case \(\kappa(S)=1\).
With the further assumption that \(S\) is minimal, we obtain the following theorem as our main result of this section.
\begin{theorem}\label{thm:main_kappa_1}
	Let \(f\colon S\to C\) be a relatively minimal elliptic fibration
	from a smooth projective surface \(S\) of Kodaira dimension \(\kappa(S)=1\).
	Then the following assertions are equivalent.
	\begin{enumerate}[wide=0pt,leftmargin=*]
		\item The tangent bundle \(T_{S}\) is pseudo-effective;
		\item The second Chern class vanishes, i.e., \(c_2(S)=0\);
		\item \(f\) is almost smooth, i.e., the only singular fibres are multiples of smooth elliptic curves.
	\end{enumerate}
\end{theorem}

We stick to \cref{note:fibration} and the following additional notation throughout this section.
\begin{notation}\label{note:kappa_1}
	\leavevmode%
	\begin{enumerate}[wide=0pt,leftmargin=*]
		\item We use the same notation as in \cref{note:fibration}.
		      In addition,  \(S\) is further assumed to be a smooth projective surface
		      of Kodaira dimension \(\kappa(S) = 1\) and \(f\) is a relatively minimal elliptic fibration,
		      i.e., free of \((-1)\)-curves among the fibres of \(f\).
		      In fact, such \(S\) is a minimal surface,
		      i.e., \(K_S\) is nef and \(c_1(S)^2=0\); see \cref{lem:rel_min_is_min}.
		\item For each point \(c\in C\),
		      we define the \emph{normalised fibre} \(\widetilde{S_c}\) over \(c\) as follows
		      \[
			      \widetilde{S_c}\coloneqq \frac{1}{12}\sum_{j\in J, f(F_j)=c}e(F_j)m_jF_j
			      +\sum_{i\in I, f(E_i)=c} \Big(1 - \Big(1- \frac{1}{12} e(S_{f(E_{i})})\Big) \nu_{i}\Big) E_{i},
		      \]
		      where \(e(\ell)\) is the Euler number of a curve \(\ell\).
		      It is clear that if \(S_c\) is a multiple of a smooth elliptic curve,
		      then the normalised multiple fibre will vanish since \(e(S_c)=0\) in this case.
		      We shall see in \cref{lem:anti_pe} that
		      all of the normalised fibre \(\widetilde{S_c}\) are \(\mathbb{Q}\)-effective
		      (cf.~\cref{rmk:normalised} for a more detailed description).
		\item By the canonical bundle formula for relatively minimal elliptic fibrations \cite[V, (12.1)~Theorem]{BHPV04}, we have
		      \[
			      \omega_S=f^{*}(\omega_C\otimes(R^1f_{*}\mathcal{O}_S)^{\vee})\otimes\mathcal{O}_S\big(\sum_j(m_j-1)F_j\big),
		      \]
		      where \(\deg(R^1f_{*}\mathcal{O}_S)^{\vee}=\chi(\mathcal{O}_S)\) (cf.~\cite[V, (12.2)~Proposition]{BHPV04}).
		      By Noether's formula, \(\chi(\mathcal{O}_S)=c_2(S)/12\), noting that \(c_1(S)^2=0\).
		\item For the relatively minimal elliptic fibration \(f\colon S\to C\),
		      recall the following invariant (cf.~\cite[V, (12.5)~Proposition]{BHPV04})
		      \[
			      \delta(f)\coloneqq \chi(\mathcal{O}_S)+(2g(C)-2+\sum_{j\in J}(1-m_j^{-1})).
		      \]
		      Then \(\kappa(S)=1\) is equivalent to \(\delta(f)>0\).
	\end{enumerate}
\end{notation}

\begin{lemma}\label{lem:rel_min_is_min}
	Let \(S\to C\) be an elliptic fibration which is relatively minimal.
	Suppose that \(\kappa(S)\geq 0\).
	Then \(S\) is minimal, i.e., the canonical divisor \(K_S\) is nef.
\end{lemma}

\begin{proof}
	Note that \(\kappa(S)\leq 1\)
	since \(S\to C\) is an elliptic fibration.
	Let \(S_{\mathrm{m}}\) be the minimal model of \(S\).
	By \cref{note:kappa_1} (3), we know that \(K_S^2=0\).
	If \(S\to S_1\) contracts some \((-1)\)-curve, then \(K_{S_1}^2>0\).
	Inductively, \(K_{S_{\mathrm{m}}}^2>0\)
	and thus \(K_{S_{\mathrm{m}}}\) is nef and big,
	a contradiction to \(\kappa(S)\leq 1\).
\end{proof}

We recall the following lemma for the convenience of later proofs.
\begin{lemma}[{\cite[\Rmnum{3}, (18.2)~Theorem and V, (12.1)~Theorem]{BHPV04}}]\label{lem:BHPV_chi_0}
	For a relatively minimal elliptic fibration \(f\colon S\to C\),
	one has \(\chi(S,\mathcal{O}_{S})\geq 0\)
	with the equality holds if and only if
	\(f\) is isotrivial and all the singular fibres are multiple of smooth elliptic curves (i.e., \(f\) is almost smooth).
\end{lemma}

\begin{lemma}\label{lem:num_linear_equiv_Y}
	We have the following linear and numerical equivalences
	\[
		\xi\sim Y+\pi^{*}T_{S/C}\equiv Y+\pi^{*}(E_{0}-\frac{c_2(S)}{12}F).
	\]
	In particular, if \(c_2(S)=0\), then the tangent bundle \(T_{S}\) is pseudo-effective.
\end{lemma}

\begin{proof}
	By the canonical bundle formula and \cref{note:fibration} (5),
	we have
	\[
		\omega_{S} = f^{*}(\omega_{C}\otimes(R^1f_{*}\mathcal{O}_{S})^{\vee}) \otimes \mathcal{O}_{S}\big(\sum_{j\in J}(m_{j}-1)F_{j}\big)
		=f^{*}(\omega_{C}\otimes(R^1f_{*}\mathcal{O}_{S})^{\vee}) \otimes \mathcal{O}_{S}(E-E_0)
	\]
	where \(\deg((R^1f_{*}\mathcal{O}_{S})^{\vee})=c_2(S)/12\).
	By \cref{note:fibration} (5) and (7), we have \(T_{S/C}\equiv E_{0}-c_2(S)F/12\).
	Together with \cref{note:fibration} (9), our lemma is thus proved.
\end{proof}

In what follows, we show the normalised fibres defined in \cref{note:kappa_1} (2) are all \(\mathbb{Q}\)-effective.
We refer readers to \cref{rmk:normalised}  for a more detailed description.
\begin{lemma}\label{lem:anti_pe}
	The divisor	\(-T_{S/C}\equiv c_2(S)F/12-E_0\) is pseudo-effective.
	In particular, the normalised fibres defined in \cref{note:kappa_1} (2)  are all \(\mathbb{Q}\)-effective.
\end{lemma}

\begin{proof}
	Applying \cite[\Rmnum{3}, (11.4)~Proposition]{BHPV04},
	we have
	\begin{align*}
		\frac{c_2(S)}{12}F - E_{0} & = \Big(\frac{1}{12} \sum_{c\in C} e(S_{c})\Big) F - E_{0} \equiv \frac{1}{12} \sum_{c\in C} e(S_{c}) S_{c}
		- \sum_{i\in I} (\nu_{i} - 1) E_{i}                                                                                                     \\
		                           & = \frac{1}{12} \sum_{j\in J} e(F_{j}) m_{j}F_{j}
		+ \frac{1}{12} \sum_{i\in I} e(S_{f(E_{i})}) \nu_{i}E_{i}
		- \sum_{i\in I} (\nu_{i} - 1) E_{i}                                                                                                     \\
		                           & = \frac{1}{12} \sum_{j\in J} e(F_j) m_j F_j
		+ \sum_{i\in I} \Big(1 - \Big(1- \frac{1}{12} e(S_{f(E_{i})})\Big) \nu_{i}\Big) E_{i}  = \sum_{c\in C}\widetilde{S_c}.
	\end{align*}
	By the non-negativity of \(e(F_j)\geq 0\), we only need to verify that
	for each non-multiple non-reduced singular fibre \(S_{f(E_i)}\) with \(i\in I\),
	we have
	\[
		1 - \Big(1- \frac{1}{12} e(S_{f(E_i)})\Big) \nu_i \geq 0.
	\]
	We refer to the Kodaira's table \cite[p.~201]{BHPV04} and check case by case, noting that
	the Euler number for a singular fibre of Types
	\Rmnum{1}\(_{0}^{*}\), \Rmnum{1}\(_{b}^{*}\), \Rmnum{2}\(^{*}\), \Rmnum{3}\(^{*}\) or \Rmnum{4}\(^{*}\)
	is \(6,b+6,10,9\) or \(8\), respectively,
	and the multiplicity of an irreducible component in a singular fibre
	of Types \Rmnum{1}\(_{0}^{*}\), \Rmnum{1}\(_{b}^{*}\), \Rmnum{2}\(^{*}\), \Rmnum{3}\(^{*}\) or \Rmnum{4}\(^{*}\)
	is no more than 2, 2, 6, 4 or 3, respectively (cf.~\cite[Table~IV.3.1]{Mir89}).
\end{proof}

To prove \cref{thm:main_kappa_1},
we first treat the case when \(f\) is non-isotrivial.
The proof of the following proposition is inspired by \cite[Proposition~5.4]{HP20}.
\begin{proposition}\label{prop:non_iso_non_pe}
	Suppose that \(f\) is non-isotrivial.
	Then the tangent bundle \(T_{S}\) is not pseudo-effective.
\end{proposition}

\begin{proof}
	Suppose to the contrary that \(T_{S}\) is pseudo-effective.
	Since \(f\) is non-isotrivial,
	there is a non-zero Kodaira-Spencer class
	which induces the following unique non-trivial extension on an elliptic curve
	\[
		0\longrightarrow\mathcal{O}_F\longrightarrow T_{S}|_{F}\cong\mathcal{F}_2\longrightarrow \mathcal{O}_{F}\longrightarrow 0.
	\]
	Note that the above short exact sequence is the restriction of
	\[
		0\longrightarrow T_{S/C}\longrightarrow T_{S}\longrightarrow f^{*}T_{C}\longrightarrow\mathcal{F}\longrightarrow 0
	\]
	to a general fibre \(F\) (cf.~\cref{note:fibration}).
	Let \(W_{F}\coloneqq \pi^{-1}(F)=\mathbb{P}(T_{S}|_{F})\), and \(C_{F}\coloneqq Y\cap W_{F}\)
	which is the section of \(\pi|_{W_{F}}\)
	associated with the surjection \(T_{S}|_{F}\to \mathcal{O}_{F}\to 0\).
	Since \(\xi\) is pseudo-effective, by the divisorial Zariski decomposition (cf.~\cite[Theorem 3.12]{Bou04}), we have
	\begin{align}\label{eq:dzd_xi}
		\xi\equiv\sum a_iY_i+P
	\end{align}
	where \(Y_i\) are finitely many prime divisors, \(a_i>0\), and \(P\) is a modified nef \(\mathbb{R}\)-divisor
	(in the sense that \(P|_D\) is pseudo-effective for every prime divisor \(D\) on \(W\))
	(cf.~\cite[Proposition~2.4]{Bou04}).
	Restricting the decomposition \eqref{eq:dzd_xi} to \(W_{F}\), we have
	\begin{align*}
		C_F\sim\xi|_{W_F}\equiv \sum_{i}a_iY_{i}|_{W_{F}}+P|_{W_{F}}.
	\end{align*}
	Here, we can choose sufficiently general \(F\) such that \(Y_{i}\cap W_{F}\) (if non-zero) are mutually distinct,
	noting that some of \(Y_i\) come from the pull-back of components of \(E\) on \(S\) and hence have no intersection with \(W_F\).
	Since the tautological section \(C_F\) is extremal in \(W_F\) but \(P|_{W_F}\) is still numerically movable by noting that \(\{W_F\}_F\) is a free family,
	we have \(P|_{W_F}=0\).
	Without loss of generality,
	we may assume that \(a_1=1\), \(Y_1\cap W_F=C_F\) and \(Y_i\cap W_F=0\) for \(i\neq 1\).
	Now that \(Y_1\cap W_F=C_F=Y\cap W_F\) for every sufficiently general fibre \(F\) of \(f\), we have \(Y_1=Y\).
	In particular, \(\pi^{*}T_{S/C}\equiv\xi-Y\equiv\sum_{i\geq 2}a_iY_i+P\) is pseudo-effective
	(cf.~\cref{lem:num_linear_equiv_Y}).
	By \cref{lem:anti_pe} and the Zariski decomposition on the surface \(S\), we have \(E_0=0\) and hence \(c_2(S)=0\), a contradiction to \cref{lem:BHPV_chi_0}.
\end{proof}

In the remaining part of this section, we shall deal with the isotrivial case, which is more troublesome.
Let us first recall the following lemma.
\begin{lemma}[cf.~e.g.~{\cite[Lemma~3.2]{PS20}}]\label{lem:singular_fibre_isotrivial}
	Let \(f\colon S\to C\) be a relatively minimal isotrivial elliptic fibration.
	Then the singular fibre of \(f\) is either a multiple of smooth elliptic curves (of Type \(_m\)\textup{\Rmnum{1}}\(_0\))
	or a non-multiple fibre not being of Types \textup{\Rmnum{1}}\(_b\) or \textup{\Rmnum{1}}\(_b^{*}\) for \(b\geq 1\).
\end{lemma}

Before moving into the proof of \cref{thm:main_kappa_1} for the isotrivial case,
we do some preparations.
We give the following remark, which gives a detailed computation
in terms of the normalised fibres and their local equations.
\begin{remark}\label{rmk:normalised}
	For every point \(c\in C\), one can calculate the normalised singular fibre \(\widetilde{S_c}\)
	(cf.~\cref{note:kappa_1}, \cref{lem:anti_pe} and \cite[Table~IV.3.1]{Mir89})
	as follows.
	\begin{align*}
		\widetilde{S_c} & = \sum_{j\in J, f(F_j)=c}\frac{1}{12}e(F_j)m_jF_j+\sum_{i\in I, f(E_i)=c}
		\Big(1-\Big(1-\frac{1}{12} e(S_{f(E_{i})})\Big) \nu_{i}\Big) E_{i}                          \\
		                & =
		\begin{cases}
			\frac{1}{6}e_1                                                                               & \text{\Rmnum{2}}       \\
			\frac{1}{4}(e_1+e_2)                                                                         & \text{\Rmnum{3}}       \\
			\frac{1}{3}(e_1+e_2+e_3)                                                                     & \text{\Rmnum{4}}       \\
			\frac{1}{2}(e_2+e_3+e_4+e_5)                                                                 & \text{\Rmnum{1}}_0^{*} \\
			\frac{5}{6}e_1+\frac{2}{3}(e_2+e_9)+\frac{1}{2}(e_3+e_7)+\frac{1}{3}(e_4+e_8)+\frac{1}{6}e_5 & \text{\Rmnum{2}}^{*}   \\
			\frac{3}{4}(e_1+e_8) + \frac{1}{2}(e_2+e_5+e_7) + \frac{1}{4}(e_3+e_6)                       & \text{\Rmnum{3}}^{*}   \\
			\frac{2}{3}(e_1+e_3+e_5) + \frac{1}{3}(e_2+e_4+e_6)                                          & \text{\Rmnum{4}}^{*}
		\end{cases}
	\end{align*}

	\begin{figure}[htbp]\label{fig:singular_fibre_iso}
		\begin{tikzpicture}[x=0.75pt,y=0.71pt,yscale=-1,  scale=0.7]
			\draw (55,0) .. controls (45,18) and (40,24) .. (35,25)
			.. controls (29,25) and (10,20) .. (10,40) .. controls (10,59) and (29,54) .. (35,55)
			.. controls (40,55) and (44,54) .. (55,75);

			\draw (55,25) node [font=\tiny] {\(e_1\)};

			\draw (40,80) node [font=\footnotesize] {\Rmnum{1}\(_{0}\)};

			\draw (50,140) .. controls (45,160) and (20,165) .. (20,165)
			.. controls (20,165) and (45,170) .. (50,190);

			\draw (60,149) node [font=\tiny] {\(e_1\)};
			\draw (10,166) node [font=\tiny] {pt};

			\draw (40,198) node [font=\footnotesize] {\Rmnum{2}};

			\draw (0,265) .. controls (25,284) and (55,284) .. (80,265);
			\draw (0,295) .. controls (25,275) and (55,275) .. (80,295);

			\draw (90,265) node [font=\tiny] {\(e_1\)};
			\draw (90,295) node [font=\tiny] {\(e_2\)};

			\draw (40,305) node [font=\footnotesize] {\Rmnum{3}};

			\draw (25,365) -- (55,425);
			\draw (25,425) -- (55,365);
			\draw (10,395) -- (70,395);

			\draw (23,359) node [font=\tiny] {\(e_1\)};
			\draw (57,359) node [font=\tiny] {\(e_2\)};
			\draw (77,395) node [font=\tiny] {\(e_3\)};

			\draw (40,440) node [font=\footnotesize] {\Rmnum{4}};

			\draw (300,10) -- (300,70); 
			\draw (340,10) -- (340,70);
			\draw (380,10) -- (380,70);
			\draw (420,10) -- (420,70);

			\draw (270,50) -- (450,50); 

			\draw (300,4) node [font=\tiny] {\(e_2\)};
			\draw (340,4) node [font=\tiny] {\(e_3\)};
			\draw (380,4) node [font=\tiny] {\(e_4\)};
			\draw (420,4) node [font=\tiny] {\(e_5\)};

			\draw (278,44) node [font=\tiny] {\(2e_1\)};

			\draw (360,85) node [font=\footnotesize] {\Rmnum{1}\(_{0}^{*}\)};

			\draw (260,120) -- (330,120); 
			\draw (350,120) -- (460,120); 

			\draw (300,160) -- (380,160); 
			\draw (420,160) -- (460,160); 

			\draw (280,100) -- (280,180); 
			\draw (320,100) -- (320,180); 
			\draw (360,100) -- (360,180); 
			\draw (400,100) -- (400,180); 
			\draw (440,100) -- (440,180); 

			\draw (268,114) node [font=\tiny] {\(2e_2\)};
			\draw (452,114) node [font=\tiny] {\(6e_6\)};

			\draw (310,154) node [font=\tiny] {\(4e_4\)};
			\draw (428,154) node [font=\tiny] {\(2e_9\)};

			\draw (280,190) node [font=\tiny] {\(e_1\)};
			\draw (320,190) node [font=\tiny] {\(3e_3\)};
			\draw (360,190) node [font=\tiny] {\(5e_5\)};
			\draw (400,190) node [font=\tiny] {\(3e_7\)};
			\draw (440,190) node [font=\tiny] {\(4e_8\)};

			\draw (360,210) node [font=\footnotesize] {\Rmnum{2}\(^{*}\)};

			\draw (260,250) -- (340,250); 
			\draw (380,250) -- (460,250);

			\draw (300,290) -- (420,290); 

			\draw (280,230) -- (280,310); 
			\draw (320,230) -- (320,310); 
			\draw (360,230) -- (360,310); 
			\draw (400,230) -- (400,310); 
			\draw (440,230) -- (440,310); 

			\draw (268,244) node [font=\tiny] {\(2e_2\)};
			\draw (452,244) node [font=\tiny] {\(2e_7\)};

			\draw (308,284) node [font=\tiny] {\(4e_4\)};

			\draw (280,320) node [font=\tiny] {\(e_1\)};
			\draw (320,320) node [font=\tiny] {\(3e_3\)};
			\draw (360,320) node [font=\tiny] {\(2e_5\)};
			\draw (400,320) node [font=\tiny] {\(3e_6\)};
			\draw (440,320) node [font=\tiny] {\(e_8\)};

			\draw (360,340) node [font=\footnotesize] {\Rmnum{3}\(^{*}\)};

			\draw (280,375) -- (320,375); 
			\draw (340,375) -- (380,375);
			\draw (400,375) -- (440,375);

			\draw (300,355) -- (300,425); 
			\draw (360,355) -- (360,425);
			\draw (420,355) -- (420,425);

			\draw (260,405) -- (460,405); 

			\draw (285,369) node [font=\tiny] {\(e_1\)};
			\draw (345,369) node [font=\tiny] {\(e_3\)};
			\draw (405,369) node [font=\tiny] {\(e_5\)};

			\draw (300,435) node [font=\tiny] {\(2e_2\)};
			\draw (360,435) node [font=\tiny] {\(2e_4\)};
			\draw (420,435) node [font=\tiny] {\(2e_6\)};

			\draw (268,399) node [font=\tiny] {\(3e_7\)};

			\draw (360,450) node [font=\footnotesize] {\Rmnum{4}\(^{*}\)};
		\end{tikzpicture}
		\caption{Singular fibre of an isotrivial elliptic fibration}
	\end{figure}
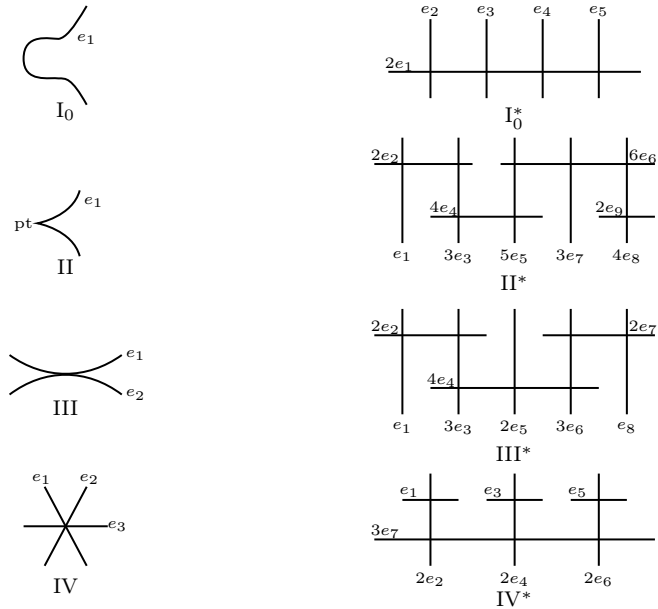

	For every point \(s\in \Gamma\),
	we may use \(x,y\) as regular parameters of \(\mathcal{O}_{S,s}\).
	The ideal sheaf of the singular points on each fibre could be chosen as below
	(cf.~\cite[Proof of Proposition~3.1]{Ser96}).
	\begin{table}[htbp]
		\centering
		\caption{Local equation of each fibre}
		\label{table:local_eq}
		\begin{tabular}{c|c|c|c|c|c|c|c}
			\toprule
			type of \(S_c\)            & \Rmnum{2}   & \Rmnum{3}   & \Rmnum{4}     & \Rmnum{1}\(_0^{*}\)                          & \Rmnum{2}\(^{*}\) & \Rmnum{3}\(^{*}\) & \Rmnum{4}\(^{*}\) \\ \midrule
			local equation of \(S_c\)  & \(x^3=y^2\) & \(x^4=y^2\) & \(x^3=y^3\)   & \multicolumn{4}{c}{\(x^{\nu_i}y^{\nu_j}=0\)}                                                             \\ \midrule
			\(\mathcal{I}_{\Gamma,s}\) & \((x^2,y)\) & \((x^3,y)\) & \((x^2,y^2)\) & \multicolumn{4}{c}{\((x,y)\)}                                                                            \\
			\bottomrule
		\end{tabular}
	\end{table}
\end{remark}

The following observation plays a significant role to the proof of \cref{thm:main_kappa_1}.
Once \(f\) is not almost smooth and if we assume \(\xi\) is pseudo-effective,
then the prime divisor \(Y\) would appear in the negative part
of the divisorial Zariski decomposition of \(\xi\).
\begin{lemma}\label{lem:xi_to_Y_nonpe}
	Let \(f\) be an isotrivial fibration.
	If \(c_2(S)>0\),
	then \(\xi|_Y\) is not pseudo-effective.
\end{lemma}

\begin{proof}
	By \cref{note:fibration} (9), \(\xi|_Y\sim -\Exc(\pi|_Y)-(\pi|_Y)^{*}(f^{*}K_C+E)\).
	If the genus \(g=g(C)\geq 1\), then \(f^{*}K_C+E\equiv(2g-2)F+E\ge 0\) is effective.
	Since \(F\) is free and \((\pi|_Y)^{*}E\) is the sum of the proper transform and some \((\pi|_Y)\)-exceptional curve,
	it follows that \((\pi|_Y)^{*}(f^{*}K_C+E)\) is effective.
	In particular, \(\xi|_Y\) is anti-pseudo-effective.
	Suppose that \(\xi|_Y\) is pseudo-effective.
	Then we have \(\xi|_Y\equiv 0\), which in turn implies  \(g(C)=1\) and \(\Exc(\pi|_Y)=0\).
	In other words, \(\Gamma=\emptyset\), i.e., the reduced structure of every fibre is smooth.
	This leads to a contradiction to \(\chi(S,\mathcal{O}_S)=c_2(S)/12>0\) (cf.~\cref{lem:BHPV_chi_0}).

	In the following, we may assume that \(C\simeq \mathbb{P}^1\).
	By \cite[Corollaries~3.13 and 3.19]{HP20},
	\(\xi|_Y\) being not pseudo-effective is equivalent to the vanishing
	\[
		H^{0}(S,\mathcal{I}_{\Gamma}^k\otimes\mathcal{O}_{S}(-k(f^{*}K_C+E))) =
		H^{0}(S,\mathcal{I}_{\Gamma}^k\otimes\mathcal{O}_{S}(k(K_S-f^{*}K_C-E-K_S))) = 0
	\]
	for all \(k\in \mathbb{N}\).
	Since \(C\simeq \mathbb{P}^1\),
	by the canonical bundle formula,
	\(K_S-f^{*}K_C-E\sim c_2(S)F/12-E_0\).
	Note that \(K_S=\delta(f)F\) with \(\delta(f)>0\) (cf.~\cref{note:kappa_1} (4)).
	Since \(f\) is isotrivial, by \cref{lem:singular_fibre_isotrivial}, the only multiple fibres are of Type \(_m\)\Rmnum{1}\(_0\),
	and hence has Euler number zero.
	Then
	\[
		\frac{c_2(S)}{12}F-E_0=\sum_{i\in I} \left(1-\left(1-\frac{e(S_{f(E_i)})}{12}\right) \nu_i\right) E_i.
	\]
	Now we obtain
	\[
		K_S-f^{*}K_C-E\sim_{\mathbb{Q}}
		\sum_{i\in I} \left(1-\left(1-\frac{e(S_{f(E_i)})}{12}\right) \nu_i\right) E_i=\sum_{c\in D}\widetilde{S_c}.
	\]
	Here, the set \(D\subseteq C\) consists of the point \(c\)
	such that the reduction of the scheme theoretic fibre \(S_c\) is not an elliptic curve.
	We are left to show that for every \(s\in \Gamma\) and every \(k\in\mathbb{N}\),
	the local equation of \(k\widetilde{S_{f(s)}}\) does not kill the ideal sheaf \(\mathcal{I}_{\Gamma,s}^k\).
	By \cref{lem:singular_fibre_isotrivial},
	the non-multiple fibres are of Kodaira's type
	\Rmnum{2}, \Rmnum{3}, \Rmnum{4},
	\Rmnum{1}\(_{0}^{*}\), \Rmnum{2}\(^{*}\), \Rmnum{3}\(^{*}\) and \Rmnum{4}\(^{*}\).
	Now the proof is finished by \cref{rmk:normalised,table:local_eq}.
\end{proof}

\begin{lemma}\label{lem:T_pe_YY_not}
	If \(T_S\) is pseudo-effective and \(Y|_Y\) is not pseudo-effective,
	then \(c_2(S)=0\).
\end{lemma}

\begin{proof}
	By \cref{prop:non_iso_non_pe}, our \(f\) is isotrivial.
	Suppose to the contrary that \(c_2(S)>0\).
	Applying \cref{lem:BHPV_chi_0,lem:xi_to_Y_nonpe}, we see that \(\xi|_Y\) is not pseudo-effective.
	By the divisorial Zariski decomposition (cf.~\cite[Theorem~3.12]{Bou04}),
	there exists \(c>0\) such that
	\begin{align}\label{eq:zdz_xi_cy}
		\xi\equiv cY+\sum a_iY_i+P,
	\end{align}
	where \(a_i>0\), \(Y_i\,(\neq Y)\) are prime divisors
	and \(P\) is a modified nef \(\mathbb{R}\)-divisor (cf.~\cite[Proposition 2.4]{Bou04}).
	Suppose first that \(c\geq 1\).
	Then \(\xi-Y\) is pseudo-effective.
	Together with \cref{lem:anti_pe},
	we have \(\xi\equiv Y\) and \(E_{0}\equiv c_2(S)F\)
	which implies that \(c_2(S)=0\),
	noting that \(E_{0}^2<0\) once it is non-zero  (cf.~\cite[\Rmnum{3}, (8.2)~Lemma]{BHPV04}),
	while \(F\) is free.
	Suppose now that \(c<1\).
	Then we have
	\[
		(1-c)Y\equiv(\xi-cY)-\pi^{*}T_{S/C}.
	\]
	Restricting to \(Y\) itself, we have
	\[
		(1-c)Y|_{Y}\equiv (\xi-cY)|_{Y}-\pi^{*}T_{S/C}|_{Y},
	\]
	where \((\xi-cY)|_{Y}\) is pseudo-effective by \cref{eq:zdz_xi_cy}.
	Since \(-T_{S/C}\) is pseudo-effective (cf.~\cref{lem:anti_pe}),
	it follows that \(-\pi^{*}T_{S/C}|_{Y}\) is also pseudo-effective, noting that \(Y\) is horizontal and irreducible.
	In particular, \((1-c)Y|_{Y}\) is pseudo-effective.
	However, this gives rise to a contradiction to our assumption.
\end{proof}

In the view of \cref{lem:T_pe_YY_not}, we are left to show that  \(Y|_Y\) is not pseudo-effective.
For the convenience of our later use, we formulate the following lemma which is a direct consequence of Zariski's lemma.

\begin{lemma}\label{lem:part_fibre_pe}
	Let \(f\colon S\to C\) be a fibration from a smooth projective surface to a smooth curve.
	Let
	\[
		S_c\coloneqq f^{-1}(c)=\sum_{i_c\in I_c}m_{i_c}E_{i_c}+\sum_{j_c\in J_c} n_{j_c}F_{j_c}
	\]
	be the fibre of \(f\) over \(c\in C\) where \(E_{i_c}\) and \(F_{j_c}\) are (distinct) irreducible components,
	and \(m_{i_c}\) and \(n_{j_c}\) are the corresponding multiplicities in \(S_c\).
	Then
	\[
		M\coloneqq\sum_{c\in D}\Big(\sum_{i_c\in I_c} a_{i_c}E_{i_c}-\sum_{j_c\in J_c} b_{j_c}F_{j_c}\Big)
	\]
	is not pseudo-effective for
	any non-empty finite set \(D\subset C\) and
	any rational numbers \(a_{i_c}\geq 0\), \(b_{j_c}>0\) satisfying the following two conditions:
	\begin{enumerate}
		\item there is at least one \(c\in D\) such that \(J_c\neq\emptyset\);
		\item for each \(c\in D\), if \(J_c=\emptyset\),
		      then there exists some \(a_{i_c}=0\) with \(i_c\in I_c\).
	\end{enumerate}

	In particular,
	\(M-\delta F\) is not pseudo-effective
	for any rational number \(\delta>0\)
	if \(M\) satisfies condition (2),
	where \(F\) is a general fibre of \(f\).
\end{lemma}

\begin{proof}
	Suppose to the contrary that \(M\) is pseudo-effective.
	Let
	\[
		M\equiv P+N=P+N_0+\sum_{c\in D}\sum_{i_c\in I_c} d_{i_c}E_{i_c}
	\]
	be the Zariski decomposition,
	where \(P\) is nef,
	\(N\) is effective (\(d_{i_c}\geq 0\)) such that \(N_0\)
	and \(E_{i_c}\)'s do not have any common components.
	Then we have
	\[
		\sum_{c\in D}\sum_{i_c\in I_{c,1}} a_{i_c}'E_{i_c} =
		P+N_0+\sum_{c\in D}\Big(\sum_{i_c\in I_{c,2}} d_{i_c}'E_{i_c}+\sum_{j_c\in J_c}b_{j_c}F_{j_c}\Big)
	\]
	with \(a_{i_c}\geq a_{i_c}'>0\),
	\(d_{i_c}'>0\), \(I_{c,1}\cap I_{c,2}=\emptyset\)
	and \(I_{c,1}\cup I_{c,2}\subseteq I_c\)
	for each \(c\in D\).
	Note that there exists at least one, say \(c'\in D\),
	such that \(I_{c',1}\neq\emptyset\),
	since the RHS is not trivial by (1).
	Take the intersection with
	\[
		\sum_{c\in D}\sum_{i_c\in I_{c,1}} a_{i_c}'E_{i_c}
	\]
	on both sides.
	Then we get the contradiction by Zariski's lemma (cf.~\cite[\Rmnum{3}, (8.2)~Lemma]{BHPV04}):
	\[
		\begin{gathered}
			\Big(\sum_{c\in D}\sum_{i_c\in I_{c,1}} a_{i_c}'E_{i,c}\Big)^2<0, \\
			\Big(P+N_0+\sum_{c\in D}\Big(\sum_{i_c\in I_{c,2}} d_{i_c}'E_{i_c}+\sum_{j_c\in J_c}b_{j_c}F_{j_c}\Big)\Big)
			\Big(\sum_{c\in D}\sum_{i_c\in I_{c,1}} a_{i_c}'E_{i,c}\Big)\geq 0,
		\end{gathered}
	\]
	noting that \(\sum_{i_{c}\in I_{c,1}} a_{i_{c}}'E_{i_{c}}\) is never a (positive) rational multiple of \(S_{c}\) by condition (2).

	For the last assertion,
	note that \(F\) is numerically equivalent to any fibre \(S_c\).
	It is not hard to see that \(M-\delta F\)
	satisfies both (1) and (2) whenever \(M\) satisfies (2).
\end{proof}

Now we come to the last part of the proof of the isotrivial case for \cref{thm:main_kappa_1}.

\begin{lemma}\label{lem:YY_non_pef}
	Suppose that \(f\) is isotrivial.
	If \(c_2(S)>0\),
	then \(Y|_{Y}\) is not pseudo-effective.
\end{lemma}

\begin{proof}
	Since \(Y\equiv \xi+\pi^{*}(c_2(S)F/12 - E_0)\) (cf.~\cref{lem:num_linear_equiv_Y}),
	we need to calculate
	\[
		\Big(\xi + \pi^{*}\Big(\frac{c_2(S)}{12}F - E_0\Big)\Big)|_Y=\xi|_Y+(\pi|_Y)^{*}\Big(\frac{c_2(S)}{12}F-E_0\Big)
	\]
	and show that it is not pseudo-effective on \(Y\).
	Recall that \(Y\)
	is isomorphic to the blow-up of the ideal sheaf \(\mathcal{I}_{\Gamma}\).
	By \cref{note:fibration} (9), \(\mathcal{O}_{W}(1)|_{Y}\simeq\mathcal{O}_{Y}(1)\simeq\mathcal{O}_{Y}(-E_{Y})\otimes (\pi|_Y)^{*}\mathcal{O}_S(-f^{*}K_C-E)\),
	where \(E_Y\) is the exceptional divisor of \(\pi|_Y\colon Y\to S\).
	In other words, \(\xi|_Y\sim -E_Y-(\pi|_Y)^{*}(f^{*}K_C+E)\).
	Therefore,
	\begin{align}\label{eq:Y_res_to_Y}
		Y|_Y & \equiv \Big(\xi + \pi^{*}\Big(\frac{c_2(S)}{12}F - E_0\Big)\Big)|_{Y} \nonumber                                     \\
		     & \sim -E_Y-(\pi|_Y)^{*}(f^{*}K_C+E)+(\pi|_Y)^{*}\Big(\frac{c_2(S)}{12}F-E_{0}\Big) \nonumber                         \\
		     & = -E_Y + (\pi|_Y)^{*}\Big(\Big(\frac{c_2(S)}{12}-(2g(C)-2)-\sum_{j\in J}(1-\frac{1}{m_j})\Big)F-2E_0\Big) \nonumber \\
		     & = -E_Y + (\pi|_Y)^{*}\Big(2\sum_{c\in C} \widetilde{S_c}-\delta(f)F\Big) \nonumber
	\end{align}
	where the last equality is due to \cref{lem:anti_pe} and \(\delta(f)>0\) is defined in \cref{note:kappa_1} (4).
	We use \cref{rmk:normalised} and the notations therein.
	The pullback \((\pi|_Y)^{*}\widetilde{S_c}\) via the blow-up \(\pi|_Y\) equals
	\begin{align*}
		(\pi|_Y)_{*}^{-1}(\widetilde{S_c}) +
		\begin{cases}
			\frac{1}{3}Y_1                                                                                                                      & \text{\Rmnum{2}}       \\
			\frac{1}{2}Y_{1,2}                                                                                                                  & \text{\Rmnum{3}}       \\
			\frac{1}{2}Y_{1,2,3}                                                                                                                & \text{\Rmnum{4}}       \\
			\frac{1}{2}(Y_{1,2}+Y_{1,3}+Y_{1,4}+Y_{1,5})                                                                                        & \text{\Rmnum{1}}_0^{*} \\
			\frac{3}{2}Y_{1,2}+\frac{7}{6}Y_{2,3}+\frac{5}{6}Y_{3,4}+\frac{1}{2}(Y_{4,5}+Y_{6,7})+\frac{1}{3}Y_{6,8}+\frac{1}{6}Y_{5,6}+Y_{8,9} & \text{\Rmnum{2}}^{*}   \\
			\frac{5}{4}(Y_{1,2}+Y_{7,8}) + \frac{3}{4}(Y_{2,3} + Y_{6,7}) + \frac{1}{2}Y_{4,5} + \frac{1}{4}(Y_{3,4} + Y_{4,6})                 & \text{\Rmnum{3}}^{*}   \\
			Y_{1,2} + Y_{3,4} + Y_{5,6} + \frac{1}{3}(Y_{2,7} + Y_{4,7} + Y_{6,7})                                                              & \text{\Rmnum{4}}^{*}
		\end{cases}
	\end{align*}
	where \((\pi|_Y)_{*}^{-1}(\widetilde{S_c})\) is the proper transform
	and \(Y_{i,j}\) is the exceptional divisor over \(e_i\cap e_j\) (scheme-theoretically).
	\begin{remark}
		We remind readers that,
		when the blown-up point \(s\) lies in the fibre of
		Types \Rmnum{1}\(_0^{*}\), \Rmnum{2}\(^{*}\), \Rmnum{3}\(^{*}\) or \Rmnum{4}\(^{*}\),
		the ideal sheaf \(\mathcal{I}_{\Gamma,s}\) is reduced and the corresponding blow-up is the usual one;
		hence the coefficient of \((\pi|_Y)^{*}\widetilde{S_c}\) along \(Y_{i,j}\) is
		simply the sum of the coefficients of \(\widetilde{S_c}\) along \(e_i\) and \(e_j\).
		However, the blown-up point in the fibre of
		Types \Rmnum{2}, \Rmnum{3} or \Rmnum{4} is non-reduced (cf.~\cref{table:local_eq}),
		in which cases, the coefficient of \((\pi|_Y)^{*}\widetilde{S_c}\) along \(Y_{i,j}\) is a bit more involved.
		We refer to \cref{rmk:type_4} for the explicit calculation on the case of Type \(\text{\Rmnum{4}}\)
		for the convenience of the reader but skip other cases for the organisation of our paper.
	\end{remark}
	Now, we come back to the proof of our lemma.
	By the above list, within each singular fibre \(S_c\),
	there exists a pair of indices \(i_1\neq i_2\in I\) such that \(e_{i_1}\cap e_{i_2}\neq \emptyset\) and the coefficient
	\[
		\coeff_{(\pi|_Y)^{*}(\widetilde{S_c})} Y_{i_1,i_2} \leq \frac{1}{2}.
	\]
	Therefore,
	\[
		2(\pi|_Y)^{*}\widetilde{S_c} - \sum_{s\in \Gamma\cap S_c} Y_{s}
	\]
	is a linear combination of some components of the fibre of \(Y\to C\) over \(c\in C\)
	satisfying (2) of \cref{lem:part_fibre_pe}.
	Noting that \(E_Y=\sum_{s\in \Gamma}Y_s\)
	and \(\delta(f)>0\),
	we see that the divisor
	\[
		Y|_Y \equiv \xi|_Y+(\pi|_Y)^{*}\Big(\frac{c_2(S)}{12}F-E_0\Big)
        \sim_{\mathbb{Q}} \sum_{c\in C} \Big(-\sum_{s\in\Gamma\cap S_c} Y_{s}+2(\pi|_Y)^{*}\widetilde{S_c}\Big) - \delta(f)(\pi|_Y)^{*}F
	\]
	is not pseudo-effective by \cref{lem:part_fibre_pe}.
	Our lemma is thus proved.
\end{proof}

\begin{example}\label{rmk:type_4}
	In this example, we calculate the blow-up of the non-reduced point in the fibre
	of Type \(\text{\Rmnum{4}}\) (cf.~\cref{table:local_eq}).
	Let \(f(x,y)=(y-x)(y-\zeta x)(y-\zeta^2x)=y^3-x^3\) be the union of three lines on \(\mathbb{A}^2\),
	where \(\zeta=\exp(2i\pi/3)\) is the primitive root of unity.
	Then the partial derivatives \(f_x=-3x^2\) and \(f_y=3y^2\).
	So the blown-up ideal is \(\mathcal{I}=(x^2,y^2)\) (cf.~\cite[Proof of Proposition 3.1]{Ser96}).
	Let \([a:b]\) be the homogeneous coordinates of \(\mathbb{P}^1\).
	Then the blow-up \(\pi\) of \(\mathcal{I}\) is defined by \(y^2a-x^2b=0\) in \(\mathbb{A}^2\times\mathbb{P}^1\).
	Consider the pullback of \(L\), \(L_1\) and \(L+L_1\),
	where \(L\coloneqq \{y-x=0\}\) is a component of \(f(x,y)=0\),
	and \(L_1\coloneqq \{y+x=0)\}\).
	These pullbacks are defined by
	\[
		(\pi|_Y)^{*}L:
		\begin{cases}
			y^2a-x^2b=0, \\
			y-x=0,
		\end{cases}
		\quad
		(\pi|_Y)^{*}L_1:
		\begin{cases}
			y^2a-x^2b=0, \\
			y+x=0,
		\end{cases}
		\text{and} \quad
		(\pi|_Y)^{*}(L+L_1):
		\begin{cases}
			y^2a-x^2b=0, \\
			y^2-x^2=0.
		\end{cases}
	\]
	Their proper transforms are
	\[
		\widetilde{L}:
		\begin{cases}
			y^2a-x^2b=0, \\
			a-b=0,       \\
			y-x=0,
		\end{cases}
		\widetilde{L_1}:
		\begin{cases}
			y^2a-x^2b=0, \\
			a-b=0,       \\
			y+x=0,
		\end{cases}
		\text{and} \quad
		\widetilde{L}+\widetilde{L_1}:
		\begin{cases}
			y^2a-x^2b=0, \\
			a-b=0,
		\end{cases}
	\]
	Note that \(\widetilde{L}+\widetilde{L_1}\)
	is a Cartier divisor on the blow-up of \(\mathcal{I}\).
	The \(\pi\)-exceptional divisor \(E\) is defined by
	\[
		\begin{cases}
			y^2a-x^2b=0, \\
			x^2=0,       \\
			y^2=0,
		\end{cases}
		\quad \text{or just} \quad
		\begin{cases}
			y^2a-x^2b=0, \\
			x^2=0,
		\end{cases}
		\text{on the affine chart } a\neq 0.
	\]
	On the affine chart \(a\neq 0\) (equivalently, \(a=1\)),
	we may calculate the lengths:
	\begin{align*}
		 & \mathbb{C}[x,y,b]/(y^2-x^2,1-b,y-x,x^2) \simeq \mathbb{C}[x]/(x^2) \quad \text{has length } 2;   \\
		 & \mathbb{C}[x,y,b]/(y^2-x^2,1-b,y+x,x^2) \simeq \mathbb{C}[x]/(x^2) \quad \text{has length } 2;   \\
		 & \mathbb{C}[x,y,b]/(y^2-x^2,1-b,x^2) \simeq \mathbb{C}[x,y]/(x^2,y^2) \quad \text{has length } 4.
	\end{align*}
	By \cite[Proposition~7.1]{Ful98}, we have
	\[
		\widetilde{L}.E\leq 2, \quad
		\widetilde{L_1}.E\leq 2, \quad
		(\widetilde{L}+\widetilde{L_1}).E=4.
	\]
	Therefore, \(\widetilde{L}.E=\widetilde{L_1}.E=2\).
	Assume the pullback of \(L\) is \(\pi^{*}L=\widetilde{L}+tE\).
	Then it follows from the projection formula that
	\[
		0=(\widetilde{L}+tE)\cdot E=2+tE^2.
	\]
	By \cite[Page 79, the paragraph before Example 4.3.1]{Ful98},
	the anti-self-intersection \(-E^2\) is the (Samuel) multiplicity of the blown-up point, which is \(4\).
	Therefore, we have \(t=1/2\)
	and hence \(\pi^{*}L=L+1/2E\).
\end{example}


\begin{corollary}\label{cor:iso_pe_t}
	Suppose that \(f\) is isotrivial.
	If \(T_S\) is pseudo-effective,
	then \(c_2(S)=0\).
\end{corollary}

\begin{proof}
	Assume that \(c_2(S)>0\).
	Then \(Y|_{Y}\) is not pseudo-effective by \cref{lem:YY_non_pef}.
	On the other hand, by \cref{lem:T_pe_YY_not},
	we have \(c_2(S)=0\), a contradiction.
\end{proof}

\begin{proof}[Proof of \cref{thm:main_kappa_1}]
	The equivalence of \((2)\) and \((3)\) follows from \cref{lem:BHPV_chi_0}.
	The implication \((2)\Rightarrow (1)\) is proved in  \cref{lem:num_linear_equiv_Y}.
	The implication \((1)\Rightarrow (2)\) follows from \cref{prop:non_iso_non_pe} and \cref{cor:iso_pe_t}.
\end{proof}

In the last part of this section, we study the Kodaira dimension \(\kappa(\mathbb{P}(T_S),\mathcal{O}(1))\).

\begin{lemma}\label{lem:pe_non_neg_kappa}
	If \(T_S\) is pseudo-effective, then \(\mathcal{O}_{\mathbb{P}(T_S)}(1)\) is
	\(\mathbb{Q}\)-linearly equivalent to an effective divisor.
	In particular, \(\kappa(\mathbb{P}(T_S),\mathcal{O}(1))\geq 0\).
\end{lemma}
\begin{proof}
	By \cref{lem:num_linear_equiv_Y},
	we know that \(\xi\sim Y+\pi^{*}T_{S/C}\)
	and \(T_{S/C}\equiv E_0-\frac{1}{12}c_2(S)F\).
	Since \(T_S\) is pseudo-effective,
	applying \cref{thm:main_kappa_1},
	we have \(c_2(S)=0\) and hence \(E_0=0\);
	in particular, \(T_{S/C}\equiv 0\).
	We shall show that \(T_{S/C}\sim_{\mathbb{Q}}0\).
	Indeed, in the view of \cref{note:fibration} (7), \cref{note:kappa_1} (3) and Serre duality, we have
	\[
		T_{S/C}=-K_S+f^{*}K_C+\sum(m_j-1)F_j+ E_0=-f^{*}((R^1f_{*}\mathcal{O}_S)^{\vee})=-f^{*}(f_{*}\omega_{S/C}).
	\]
	Since \(\deg(f_{*}\omega_{S/C})=\chi(S,\mathcal{O}_S)=(c_1(S)^2+c_2(S))/12=0\),
	by \cite[\Rmnum{3}, (18.3)~Proposition]{BHPV04}, \(f_{*}\omega_{S/C}\)
	and hence \(T_{S/C}\) are \(\mathbb{Q}\)-trivial.
	Consequently, \(\xi\sim_{\mathbb{Q}} Y\) and thus \(\kappa(\mathbb{P}(T_S),\mathcal{O}(1))\geq 0\).
\end{proof}

\begin{lemma}\label{lem:iitake_dim_locally_trivial}
	Let \(f\colon S\to C\) be a locally trivial elliptic fibration
	over a smooth projective curve \(C\) of genus \(g=g(C)\geq 1\).
	Then \(\kappa(\mathbb{P}(T_S),\mathcal{O}(1))=1-\kappa(C)\).
\end{lemma}

\begin{proof}
	Since \(\chi(S,\mathcal{O}_S)=0\),
	after replacing \(S\) by a further \'etale cover,
	we may assume that the relative tangent bundle \(T_{S/C}\) is trivial
	(cf.~\cite[\Rmnum{3}, (18.3)~Proposition]{BHPV04} and \cref{lem:ueno-kappa}).
	We consider the symmetric power of the following short exact sequence
	\[
		0\to \mathcal{O}_S\to T_{S}\to f^{*}T_C\to 0,
	\]
	noting that \(T_{S/C}=\mathcal{O}_S\) by our reduction.
	For each positive integer \(m\), we have
	\[
		0\to \Sym^{m-1}T_S\to\Sym^mT_S\to f^{*}(-mK_C)\to 0.
	\]
	Hence, we have
	\[
		h^0(S,\Sym^mT_S)\leq 1+\sum_{i=1}^m h^0(S,f^{*}(-iK_C)).
	\]
	In particular, if \(g(C)\geq 2\), then \(h^0(S,\Sym^mT_S)\leq 1\) for any \(m\).
	On the other hand, since \(f\) is smooth, it follows from \cref{thm:main_kappa_1} that \(T_S\) is pseudo-effective.
	Together with \cref{lem:pe_non_neg_kappa}, we have  \(h^0(S,\Sym^mT_S)=1\) for any \(m\) whenever \(g(C)\geq 2\).
	As a result, \(\kappa(\mathbb{P}(T_S),\mathcal{O}(1))=0\) when \(g(C)\geq 2\).

	If \(g(C)=1\), then by the vanishing of \(\delta(f)\) (cf.~\cref{note:kappa_1} (4)), we have \(\kappa(S)=0\).
	Since \(S\) is also minimal, our lemma follows from \cref{lem:minimal_kappa_0}.
\end{proof}

As a corollary, we show our main theorem when \(S\) is assumed to be minimal.

\begin{corollary}\label{cor:fibred_surf}
	Let \(S\) be a smooth minimal projective surface, i.e., \(K_S\) is nef.
	Then the tangent bundle \(T_S\) is pseudo-effective if and only if
	the second Chern class \(c_2(S)=0\).
	Moreover, after a Galois \'etale cover,
	\(S\) admits a locally trivial elliptic fibration onto a smooth projective curve \(C\) with \(\kappa(S)=\kappa(C)\).
	In particular, after a further \'etale cover,
	\(S\) is either an abelian surface or isomorphic to a product of an elliptic curve and a smooth curve of genus \(\ge 2\).
\end{corollary}

\begin{proof}
	The first statement follows from \cref{lem:minimal_kappa_0,prop:kappa_2_tx,thm:main_kappa_1}.
	Let \(f\colon S\to C\) be an elliptic fibration.
	Note that \(\kappa(S)=0\) or \(1\) and \(c_2(S)=0\) if and only if the only singular fibres of \(f\) are multiple of elliptic curves.
	If \(C\simeq \mathbb{P}^1\),
	then \(f\) has at least three multiple fibres,
	since \(\delta(f)\geq 0\) by \cite[V, (12.5)~Proposition]{BHPV04}.
	Applying \cite[Lemma~1.1.9]{GMM21},
	we get a finite Galois \'etale map \(g_S\colon S'\to S\)
	which is induced by a ramified base change \(g\colon C'\to C\)
	and \(f'\colon S'\to C'\)
	has no multiple fibre.
	That is,
	\(f'\) is a locally trivial elliptic fibration.
	Note that when \(\kappa(S)=1\),
	one has \(\delta(f')>0\) and hence \(g(C')\geq 2\);
	when \(\kappa(S)=0\), one has \(0\leq \kappa(C')\leq \kappa(S')=0\) and hence \(\kappa(C')=0\).
	Finally, the last part of our corollary follows from \cite[Corollary 26.5]{Har10}.
\end{proof}

\section{Proof of \texorpdfstring{\cref{mainthm:pe_iitaka_fib}}{Theorem}}\label{sec-pf-main}

In this section, we prove \cref{mainthm:pe_iitaka_fib}.
In the view of  \cref{prop:kappa_0_pe_minimal,prop:kappa_2_tx,thm:main_kappa_1},
our main task left is to exclude the non-minimal surface
which is of Kodaira dimension one and has pseudo-effective tangent bundle.
The following theorem is our main result of this section.
\begin{theorem}\label{thm:minimal}
	Let \(S\) be a smooth projective surface of \(\kappa(S)=1\).
	If \(T_S\) is pseudo-effective, then \(K_S\) is nef, i.e., \(S\) is minimal.
\end{theorem}

\begin{proof}
	In the view of \cref{lem:bir_descend,lem:rel_min_is_min},
	we only need to exclude the case when \(S\coloneqq S_2\) is a blow-up of
	a smooth minimal surface \(S_1\).
	Suppose to the contrary that the tangent bundle \(T_{S_2}\) is pseudo-effective.

	Let us consider the following commutative diagram as in \cref{note:elementary_transform}
	\[
		\xymatrix{
		\mathbb{P}(T_{S_1})\ar[d]^{\pi_1}
		& \mathbb{P}(h^{*}T_{S_1})\ar@{-->}[rr]\ar[d]^{\widetilde{\pi_1}}\ar[l]_{\widetilde{h}}
		& & \mathbb{P}(T_{S_2})\ar[d]^{\pi_2}\\
		S_1\ar[d]^{f_1}
		& S_2\ar[dl]^{f_2}\ar@{=}[rr]\ar[l]_{h}
		& & S_2\\
		C
		}
	\]
	where \(h\colon S_2\to S_1\) is a single blow-up between smooth projective surfaces
	with an exceptional \((-1)\)-curve \(D\),
	\(f_1\colon S_1\to C\) is an elliptic fibration onto a smooth curve \(C\),
	\(f_2\coloneqq f_1\circ h\),
	and \(\pi_i\colon \mathbb{P}(T_{S_i})\to S_i\) is the natural projection.
	Let \(\xi_i\) be the tautological divisor of the projective bundle \(\mathbb{P}(T_{S_i})\)
	and let \(Y_i\) be the divisor in \(\mathbb{P}(T_{S_i})\) defined in \cref{note:fibration} (8).
	Then we have	\(\xi_i\sim Y_i+\pi_i^{*}T_{S_i/C}\) (cf.~\cref{note:fibration} (9)).

	We apply \cref{cor:fibred_surf} to get an \'etale morphism \(S_1'\to S_1\),
	which is induced by a (possibly ramified) base change \(g\colon C'\to C\),
	such that \(S_1'\) admits a locally trivial elliptic fibration
	over a smooth curve \(C'\) of genus \(\geq 2\).
	Then we have the following commutative diagram induced by the base change
	\[
		\xymatrix{
			S_2'\ar[r]^{g_2}\ar[d]_{h'}&S_2\ar[d]^h \\
			S_1'\ar[r]^{g_1}\ar[d]_{f_1'}&S_1\ar[d]^{f_1} \\
			C'\ar[r]^g&C
		}
	\]
	where \(h'\) is the blow-up along smooth (reduced) points and \(g_2\) is \'etale.
	Then \(T_{S_2'}=g_2^{*}T_{S_2}\) is pseudo-effective (cf.~\cref{lem:ueno-kappa}).
	Replacing \(S_i\) with \(S_i'\),
	we may assume that \(f_1\) is locally trivial with \(g(C)\geq 2\).

	Since the divisor \(\xi_2\) is pseudo-effective,
	it follows from the divisorial Zariski decomposition (cf.~\cite[Theorem~3.12]{Bou04}) that
	\[
		\xi_2\equiv\sum a_i P_i+N,
	\]
	where \(a_i>0\), \(P_i\) are prime divisors and \(N\) is a modified nef \(\mathbb{R}\)-divisor.
	Since \(T_{S_2}\) is pseudo-effective,
	applying \cref{lem:bir_descend} and \cref{cor:fibred_surf},
	we know that \(c_2(S_1)=0\) (indeed, we can even make \(S_1\) a product variety);
	hence, \(T_{S_1/C}\equiv 0\)
	(cf.~\cref{lem:anti_pe,lem:num_linear_equiv_Y}).
	We calculate the difference (cf.~\cref{note:fibration} (7))
	\[
		h^{*}T_{S_1/C}-T_{S_2/C}=-h^{*}K_{S_1}+K_{S_2}+h^{*}E_{S_1}-E_{S_2}=D+h^{*}E_{S_1}-E_{S_2}=D,
	\]
	where \(E_{S_i}\) is defined in \cref{note:fibration} (5).
	Therefore, by \cref{note:fibration} (9), we have
	\[
		Y_2-\pi_2^{*}D\sim \xi_2\equiv \sum a_iP_i+N.
	\]
	Without loss of generality, we may assume that \(P_1=Y_2\) with \(a_1\) possibly being zero.
	Since \(0\neq-\pi_2^{*}D\) is anti-pseudo-effective, we see that \(0\leq a_1<1\).
	From the numerical equivalence
	\[
		(1-a_1)Y_2|_{Y_2}\equiv \sum_{i\geq 2}a_iP_i|_{Y_2}+N|_{Y_2}+(\pi_2|_{Y_2})^{*}D,
	\]
	the pseudo-effectiveness of \(Y_2|_{Y_2}\) follows, noting that \(\pi_2^*D|_{Y_2}\) is effective.
	On the other hand, we have the following equivalence (cf.~\cref{note:fibration} (9))
	\[
		Y_2|_{Y_2} \equiv (\xi_2+\pi_2^{*}D)|_{Y_2} \sim -\Exc(\pi_2|_{Y_2})+(\pi_2|_{Y_2})^{*}(-f_2^{*}K_C+D).
	\]
	Since \(g(C)\geq 2\), after moving one free fibre from \(f^{*}K_C\) to \(f_2^{-1}(f_2(D))\),
	our \(Y_2|_{Y_2}\) is anti-pseudo-effective and non-zero,
	which contradicts  \(Y_2|_{Y_2}\) being pseudo-effective.
	So we finish the proof of our theorem.
\end{proof}

\begin{proof}[Proof of \cref{mainthm:pe_iitaka_fib}]
	In the view of \cref{prop:kappa_0_pe_minimal,prop:kappa_2_tx}, we may assume that \(\kappa(S)=1\).
	Then, the first half follows from   \cref{thm:main_kappa_1,thm:minimal}.
	The second half follows from
	\cref{lem:iitake_dim_locally_trivial,lem:minimal_kappa_0,cor:fibred_surf}.
\end{proof}

\section{Ruled surface over non-rational base, \texorpdfstring{Proof of \cref{prop:uniruled_blowup}}{Proposition}}\label{sec:ruled}

In the last section, we study non-rational uniruled surfaces with pseudo-effective tangent bundles.
First, 
we show that any projective bundle (of arbitrary rank) over a smooth curve always has a pseudo-effective tangent bundle,
which slightly extends \cite{Kim22}.
\begin{proposition}[{cf.~\cite{Kim22}}]\label{prop:ruled_surfaces}
	The tangent bundle \(T_X\) of any projective
    bundle
    \(f\colon X=\mathbb{P}_{C}(\mathcal{E})\to C\) over a smooth curve \(C\) is pseudo-effective.
	In particular, \(T_{X}\) is pseudo-effective but non-big if and only if \(\mathcal{E}\) is semi-stable
	and \(C \not\simeq \mathbb{P}^1\).
\end{proposition}

\begin{proof}
	In the view of \cite{Kim22},
    we only need to show that when \(\mathcal{E}\)
    is semi-stable,
    the tangent bundle \(T_X\) is pseudo-effective.

    Now \(\mathcal{E}\) is semi-stable,
    so are \(\mathcal{E}^{\vee}\) and hence \(\mathcal{E}^{\vee}\otimes\mathcal{E}\).
    Since the determinant \(\det(\mathcal{E}^{\vee}\otimes\mathcal{E})\simeq \mathcal{O}_C\),
    the semi-stable vector bundle \(\mathcal{E}^{\vee}\otimes\mathcal{E}\) is nef by \cite[Proposition~6.4.11]{Laz04}.
	We consider the following relative Euler sequence
    \[
		0\longrightarrow \mathcal{O}_X \longrightarrow f^{*}\mathcal{E}^{\vee}\otimes\mathcal{O}_X(1) \longrightarrow T_{X/C} \longrightarrow 0,
	\]
    where the relative tangent bundle \(T_{X/C}\coloneqq\Omega_{X/C}^{\vee}\).
    Then the following composite map   (cf.~\cite[Chapter~\Rmnum{2}, Proposition~7.11(2)]{Har77})
    \[
        f^{*}(\mathcal{E}^{\vee}\otimes\mathcal{E})\longrightarrow
        f^{*}\mathcal{E}^{\vee}\otimes\mathcal{O}_X(1)\longrightarrow T_{X/C}
    \]
   is a surjection, which implies that \(T_{X/C}\) is nef (cf.~\cite[Proposition~6.1.2]{Laz04})
    and hence pseudo-effective.
	So our proposition follows from \cref{lem:pe_inj}.
\end{proof}

Now we are in the position of proving \cref{prop:uniruled_blowup}.

\begin{proof}[Proof of \cref{prop:uniruled_blowup}]
	Let \(h\colon S_2\to S_1\) be the blow-up of \(S_1=S\) along \(p\).
	We may assume that \(\mathcal{E}\) is normalised (cf.~\cite[Chapter V, Notation 2.8.1]{Har77}).
	Since \(T_{S_1}\) is pseudo-effective but not big, by \cref{prop:ruled_surfaces}, \(\mathcal{E}\) is semi-stable.
	Together with \(\mathcal{E}\) being normalised, we have  \(\deg\det\mathcal{E}\geq 0\).
	Denote by \(C_0\) the tautological divisor on \(S_1=\mathbb{P}_C(\mathcal{E})\), which is not necessarily effective.
	Then \(\overline{\NE}(S_1)=\Nef(S_1)=\mathbb{R}_{+}(2C_0-f_1^{*}\det\mathcal{E})+\mathbb{R}_{+}F\),
	where \(F\) is a fibre of \(f_1=f\colon S_1=S\to C\)
	(cf.~\cite[Proposition~3.1]{Miy87}).

	Consider the following commutative diagram as in \cref{note:elementary_transform}.
	\[
		\xymatrix{
		\mathbb{P}(T_{S_1})\ar[d]^{\pi_1}
		& \mathbb{P}(h^{*}T_{S_1})\ar@{-->}[rr]\ar[d]^{\widetilde{\pi_1}}\ar[l]_{\widetilde{h}}
		& & \mathbb{P}(T_{S_2})\ar[d]^{\pi_2}\\
		S_1\ar[d]^{f_1}&S_2\ar[dl]^{f_2}\ar@{=}[rr]\ar[l]_{h}&&S_2\\
		C
		}
	\]
	Let \(D\) be the \(h\)-exceptional \((-1)\)-curve,
	\(f_2\coloneqq f_1\circ h\)  the induced composite map,
	\(\xi_i\) the tautological divisor of the projective bundle \(\mathbb{P}(T_{S_i})\),
	and \(Y_i\) the divisor in \(\mathbb{P}(T_{S_i})\) defined in \cref{note:fibration} (8).
	Then we have (cf.~\cref{note:fibration} (9))
	\[
		\xi_i\sim Y_i+\pi_i^{*}T_{S_i/C}.
	\]

	Let us calculate the difference
	\[
		h^{*}T_{S_1/C}-T_{S_2/C}=-h^{*}K_{S_1}+K_{S_2}+h^{*}E_{S_1}-E_{S_2}=D+h^{*}E_{S_1}-E_{S_2}=D,
	\]
	where \(E_{S_i}\) is defined in \cref{note:fibration} (5).
	Here, since \(f_1\) is smooth and the only singular fibre of \(f_2\) is reduced, we have \(h^{*}E_{S_1}=E_{S_2}=0\)
	and
	\[
		T_{S_1/C}=-K_{S_1}+f_1^{*}K_C\sim 2C_0-f_1^{*}\det\mathcal{E}.
	\]

	\begin{claim}\label{claim:pe_kappa}
		The divisor \(T_{S_2/C}=h^{*}T_{S_1/C}-D\) is pseudo-effective if and only if
		there exist some positive integer \(m\) and some line bundle
		\(\mathcal{L}\equiv T_{S_1/C}\)
		such that \(H^0(S_1,\mathfrak{m}_p^m\otimes\mathcal{L}^{\otimes m})\neq0\),
		where \(\mathfrak{m}_p\) is the maximal ideal of \(\mathcal{O}_{S_1,p}\).
	\end{claim}

	\begin{proof}[Proof of \cref{claim:pe_kappa}]
		One direction is clear.
		Let us assume that \(h^{*}T_{S_1/C}-D\) is pseudo-effective.
		Then we have the Zariski decomposition \(h^{*}T_{S_1/C}-D\equiv \sum_i a_iP_i+N\),
		where \(N\) is nef and \(\sum_i a_iP_i\) has negative definite intersection matrix.
		Since \((h^{*}T_{S_1/C}-\mu D)^2<0\) for any \(\mu>0\), there exists at least one, say \(P_1\),
		which is not \(h\)-exceptional such that \(P_1^2<0\).
		Pushing this forward to \(S_1\),
		we obtain \(T_{S_1/C}\equiv \sum_i a_ih_{*}P_i+h_{*}N\).
		Since \(T_{S_1/C}\) is extremal,
		\(T_{S_1/C}\equiv th_{*}P_1\) for some \(t>0\).
		Suppose to the contrary that
		for any positive integer \(m\)
		and any line bundle
		\(\mathcal{L}\equiv T_{S_1/C}\),
		one has \(h^0(S_1,\mathfrak{m}_p^m\otimes\mathcal{L}^{\otimes m})=0\).
		By our assumption, \(h^{*}(h_{*}P_1)=P_1+sD\) with \(st<1\);
		cf.~\cite[Chapter~V, Proposition~3.6]{Har77}.
		Hence, \(h^{*}T_{S_1/C}-D\equiv tP_1-(1-st)D\) which is not pseudo-effective
		by noting that \(P_1\) is extremal and not parallel to \(D\).
		This leads to a contradiction to our assumption.
	\end{proof}

	We come back to the proof of \cref{prop:uniruled_blowup}.
	If \(H^0(S_1,\mathfrak{m}_p^m\otimes\mathcal{L}^{\otimes m})\neq 0\)
	for some positive integer \(m\) and some line bundle \(\mathcal{L}\equiv T_{S_1/C}\),
	then \(T_{S_2/C}=h^{*}T_{S_1/C}-D\) is pseudo-effective
	and thus \(T_{S_2}\) is pseudo-effective, noting that \(\xi_2\sim Y_2+\pi_2^*T_{S_2/C}\) (cf.~\cref{lem:pe_inj}).

	Now we assume that
	\(H^0(S_1,\mathfrak{m}_p^m\otimes\mathcal{L}^{\otimes m})=0\)
	for any positive integer \(m\) and any line bundle \(\mathcal{L}\equiv T_{S_1/C}\).
	Suppose that \(\xi_2\) is pseudo-effective.
	It follows from the divisorial Zariski decomposition (cf.~\cite[Theorem~3.12]{Bou04}) that
	\[
		Y_2+\pi_2^{*}T_{S_2/C}\sim \xi_2\equiv \sum a_iP_i+N,
	\]
	where \(a_i>0\), \(P_i\) are prime divisors and \(N\) is a modified nef \(\mathbb{R}\)-divisor.
	By \cref{note:fibration} (9), \(\xi_2|_{Y_2}=-\Exc(\pi_2|_{Y_2})-(\pi_2|_{Y_2})^{*}f_2^{*}K_C\),
	which is anti-pseudo-effective and non-zero.
	Without loss of generality,
	we may assume that \(P_1=Y_2\) with \(a_1>0\).
	Note that now \(\pi_2^{*}T_{S_2/C}\sim\xi_2-Y_2\)
	is not pseudo-effective by \cref{claim:pe_kappa} (cf.~\cite[Chapter~\Rmnum{2}, Lemma~5.6]{Nak04}).
	Consequently, we deduce that \(0<a_1<1\).
	Consider the following restriction
	\begin{align*}
		(\xi_2-a_1Y_2)|_{Y_2} & = ((1-a_1)\xi_2+a_1\pi_2^{*}T_{S_2/C})|_{Y_2}                                              \\
		                      & = (1-a_1)(-\Exc(\pi_2|_{Y_2})-(\pi_2|_{Y_2})^{*}f_2^{*}K_C)+a_1(\pi_2|_{Y_2})^{*}T_{S_2/C}
	\end{align*}
	which is a pseudo-effective divisor on \(Y_2\).
	Pushing this forward to \(S_2\) along the birational morphism \(\pi_2|_{Y_2}\),
	we have the pseudo-effectiveness of
	\[
		(\pi_2|_{Y_2})_{*}((\xi_2-a_1Y_2)|_{Y_2}) = -(1-a_1)f_2^{*}K_C + a_1T_{S_2/C},
	\]
	while the RHS is never pseudo-effective whenever \(0<a_1<1\) and \(g(C)\geq1\).
	In particular, our assumption is absurd and \(\xi_2\) is thus not pseudo-effective.
	So we finish the proof of our proposition.
\end{proof}

\bibliographystyle{alpha}
\bibliography{ref}

\end{document}